\documentclass[fleq]{amsart}
\usepackage{amssymb,amsmath,latexsym,amsbsy,color,enumerate}
\numberwithin{equation}{section}

\newtheorem{theorem}{Theorem}
\newtheorem{definition}[theorem]{Definition}
\newtheorem{lemma}{Lemma}
\newtheorem{remark}{Remark}
\newtheorem{proposition}{Proposition}
\newtheorem{corollary}{Corollary}

\begin{document}
 \title{Pull-back of currents by meromorphic maps}
 \author{Tuyen Trung Truong}
    \address{Indiana University Bloomington IN 47405}
 \email{truongt@indiana.edu}
\thanks{}
    \date{\today}
    \begin{abstract}
Let $X$ and $Y$ be compact K\"ahler manifolds, and let $f:X\rightarrow Y$ be a dominant meromorphic map. Base upon a regularization theorem of Dinh and
Sibony for DSH currents, we define a  pullback operator $f^{\sharp}$ for currents of bidegrees $(p,p)$ of finite order on $Y$ (and thus for {\it any}
current, since $Y$ is compact). This operator has good properties as may be expected.

Our definition and results are compatible to those of various previous works of Meo, Russakovskii and Shiffman, Alessandrini and Bassanelli, Dinh and
Sibony, and can be readily extended to the case of meromorphic correspondences.

We give an example of a meromorphic map $f$ and two nonzero positive closed currents $T_1,T_2$ for which $f^{\sharp}(T_1)=-T_2$. We use Siu's
decomposition to help further study on pulling back positive closed currents. Many applications on finding invariant currents are given.

\end{abstract}
\maketitle
\section{Introduction}       
\label{SectionIntroduction}

Let $X$ and $Y$ be two compact K\"ahler manifolds, and let $f:X\rightarrow Y$ be a dominant meromorphic map. For a $(p,p)$-current $T$ on $Y$, we seek to define
a pullback $f^{\sharp}(T)$ which has good properties. Such a pullback operator will be helpful in complex dynamics, in particular in the problem of finding
invariant closed currents for a selfmap.

We let $\pi _X,\pi _Y: X\times Y\rightarrow X,Y$ be the two projections (When $X=Y$ we denote these maps by $\pi _1$ and $\pi _2$). Let
$\Gamma _f\subset X\times Y$ be the graph of $f$, and let $\mathcal{C}_f\subset \Gamma _f$ be the critical set of $\pi _Y$, i.e. the smallest analytic subvariety
of $\Gamma _f$ so that the restriction of $\pi _Y$ to $\Gamma _f-\mathcal{C}_f$ has fibers of dimension $dim(X)-dim(Y)$. For a set $B\subset Y$, we define
$f^{-1}(B)=\pi _X(\pi _Y^{-1}(B)\cap \Gamma _f)$, and for a set $A\subset X$ we define $f(A)=\pi _Y(\pi _X^{-1}(A)\cap \Gamma _f)$.

If $T$ is a smooth form on $Y$, then it is standard to define $f^*(T)$ as a current on $X$ by the formula $f^*(T)=(\pi _X)_*(\pi _Y^*(T)\wedge [\Gamma
_f])$. This definition descends to cohomology classes: If $T_1$ and $T_2$ are two closed smooth forms on $Y$ having the same cohomology classes, then
$f^*(T_1)$ and $f^*(T_2)$ have the same cohomology class in $X$. This allows us to define a pullback operator on cohomology classes. These considerations
apply equally to continuous forms. When $T$ is an arbitrary current on $Y$, we can still define $\pi _Y^*(T)$ as a current on $X\times Y$. However, in
general it is not known how to define the wedge product of the two currents $\pi _Y^*(T)$ and $[\Gamma _f]$. This is the source of difficulty for
defining pullback for a general current.

For some important classes of currents (positive closed and positive $dd^c$-closed currents, $DSH$ currents, for definitions see the next subsection),
there have been works on this topic by Meo \cite{meo}, Russakovskii and Shiffman \cite{russakovskii-shiffman}, Alessandrini and Bassanelli
\cite{alessandrini-bassanelli2}, Dinh and Sibony \cite{dinh-sibony2},\cite{dinh-sibony4}. We will give more details on these works later, but here will
discuss only some general ideas used in these papers. Roughly speaking, in the works cited above, to define pullback of a $(p,p)$ current $T$, the
authors use approximations of $T$ by sequences of smooth $(p,p)$ forms $T_n$ satisfying certain properties, and then  define $f^{\sharp}(T)=\lim
_{n\rightarrow\infty}f^*(T_n)$ if the limit exists and is the same for all such sequences. In order to have such approximations then $T$ must have some
positive property. In these definitions, the resulting pullback of a positive current is again positive.

Our idea for pulling back a general $(p,p)$ current $T$ is as follows. Assume that we have a well-define pullback $f^{\sharp}(T)$. Then for any smooth
form of complement bidegree $\alpha$ we should have
\begin{eqnarray*}
\int _Xf^{\sharp}(T)\wedge \alpha =\int _YT\wedge f_*(\alpha ).
\end{eqnarray*}
The wedge product in the integral of the RHS is not well-defined in general. To define it we adapt the above idea, that is to use smooth approximations of either $T$ or $f_*(\alpha)$. Fortunately, since $Y$ is compact, any current $T$ is of a finite order $s$. Moreover since $f_*(\alpha )$ is a $DSH$ current, we can use the regularization theorem in \cite{dinh-sibony4} to produce approximation by $C^s$ forms $K_n(f_*(\alpha ))$ with desired properties. Then we define
\begin{eqnarray*}
\int _Xf^{\sharp}(T)\wedge \alpha =\lim _{n\rightarrow\infty}\int _YT\wedge K_n(f_*(\alpha )),
\end{eqnarray*}
if the limit exists and is the same for such good approximations. The details of this definition will be given in the next subsection. We conclude this
subsection commenting on the main results of this paper:

-Our pullback operator is compatible with the standard definition for continuous form and with the definitions in the works cited above.

-There are examples of losing positivity for currents of higher bidegrees when pulled back by meromorphic maps.

-We obtain a natural criterion on pulling back analytic varieties which, combined with Siu's decomposition, can be used to help further study pullback of
general positive closed currents.

-We can apply the definition to examples having invariant positive closed currents of higher bidegrees whose supports are contained in pluripolar sets.

\subsection{Definitions}

For convenience, let us first recall some facts about currents. The notations of positive and strongly positive currents in this paper follow the book
\cite{demailly}. For a current $T$ on $Y$, let $supp(T)$ denote the support of $T$. Given $s\geq 0$, a current $T$ is of order $s$ if it acts
continuously on the space of $C^s$ forms on $Y$ equipped with the usual $C^s$ norm. A positive $(p,p)$ current $T$ is of order $0$. If $T$ is a positive
$(p,p)$ current then its mass is defined as $||T||=<T, \omega _Y^{dim(Y)-p}>$, where $\omega _Y$ is a given K\"ahler $(1,1)$ form of $Y$. If $T$ is a
closed current on $Y$, we denote by $\{T\}$ its cohomology class. If $V$ is a subvariety in $Y$, we denote by $[V]$ the current of integration on $V$,
which is a strongly positive closed current.
 We use $\rightharpoonup$ for weak convergence of currents.

For any $p$, we define $DSH^{p}(Y)$ (see Dinh and Sibony \cite{dinh-sibony1}) to be the space of $(p,p)$ currents $T=T_1-T_2$, where $T_i$ are positive
currents, such that $dd^cT_i=\Omega _i^+-\Omega _i^-$ with $\Omega _i^{\pm}$ positive closed. Observe that $||\Omega _i^{+}||=||\Omega _i^-||$ since they
are cohomologous to each other because $dd^c(T_i)$ is an exact current. Define the $DSH$-norm of $T$ as
\begin{eqnarray*}
||T||_{DSH}:=\min \{||T_1||+||T_2||+||\Omega _1^{+}||+||\Omega _2^{+}||,~T_i,~\Omega _i,~\mbox{as above}\}.
\end{eqnarray*}
Using compactness of positive currents, it can be seen that we can find $T_i,~\Omega _i^{\pm}$ which realize $||T||_{DSH}$, hence the minimum on the RHS
of the definition of $DSH$ norm. We say that $T_n\rightharpoonup T$ in $DSH^p(Z)$ if $T_n$ weakly converges to $T$ and $||T_n||_{DSH}$ is bounded.

Our definition is modelled on th smooth approximations given by Dinh and Sibony \cite{dinh-sibony1}. However, some restrictions should be imposed on the
approximations when we deal with the case of general maps:

1) Since any definition using local approximations will give a positive current as the resulting pullback of positive currents, in general we need to use global approximations in order to deal with the cases like the map $J_X$ in Section $4$.

2)  For a general compact K\"ahler manifold, it is not always  possible to approximate a positive closed current by positive closed smooth forms (see Proposition \ref{PropositionNoVeryGoodApproximation} for an example where even the negative parts of the approximation are not bounded).

3)  The more flexible we allow in approximating currents, the more restrictive the maps and currents we can define pullback. For example, we have the following observation

\begin{lemma}Assume that for any positive closed smooth $(p,p)$ form $T$  and for every sequence of positive closed smooth forms $T_n^{\pm}$ whose masses $||T_n^{\pm}||$ are uniformly bounded and $T_n^+-T_n^-\rightharpoonup  T$, then $f^*(T_n^+-T_n^-)\rightharpoonup  f^*(T)$.  Then the same property holds for any positive closed $(p,p)$ current $T$.
\label{LemmaNotAllCanBePullbacked}\end{lemma}
\begin{proof}
In fact, let $T_n^+-T_n^-$  and $S_n^+-S_n^-$ be two sequences weakly converging to a positive closed $(p,p)$ current $T$, where $T_n^{\pm}$  and
$S_n^{\pm}$ are positive closed smooth $(p,p)$ forms having uniformly bounded masses. Then $(T_n^++S_n^-)-(T_n^-+S_n^+)$ is a sequence weakly converges
to $0$ with the same property, and because $0$ is a smooth form, we must have $f^*(T_n^++S_n^-)-f^*(T_n^-+S_n^+)$ weakly converges to $0$ by assumption.
Hence $f^*(T_n^+-T_n^-)$ and $f^*(S_n^+-S_n^-)$ converges to the same limit.
\end{proof}
Roughly speaking, under the conditions of Lemma \ref{LemmaNotAllCanBePullbacked} then all positive closed currents can be pulled back. However, this is
not true in general (see Example $2$). We will restrict to use only good approximation schemes, defined as follows

\begin{definition}
Let $Y$ be a compact K\"ahler manifold. Let $s\geq 0$ be an integer. We define a good approximation scheme by $C^s$ forms for $DSH$ currents on $Y$ to be an assignment that for a $DSH$ current $T$ gives    two sequences $K_n^{\pm}(T)$ (here $n=1,2,\ldots $) where $K_n^{\pm}(T)$ are $C^s$ forms of the same bidegrees as $T$, so that $K_n(T)=K_n^+(T)-K_n^-(T)$  weakly converges to $T$, and moreover the following properties are satisfied:

1) Boundedness: The $DSH$ norms of  $K_n^{\pm}(T)$ are uniformly bounded.

2) Positivity: If $T$ is positive then $K_n^{\pm}(T)$ are positive, and $||K_n^{\pm}(T)||$ is uniformly bounded with respect to n.

3) Closedness: If $T$ is positive closed then $K_n^{\pm}(T)$ are positive closed.

4) Continuity: If $U\subset Y$ is an open set so that $T|_U$ is a continuous form then $K_n^{\pm}(T)$ converges locally uniformly on $U$.

5) Additivity: If $T_1$ and $T_2$ are two $DSH^p$ currents, then $K_n^{\pm}(T_1+T_2)=K_n^{\pm}(T_1)+K_n^{\pm}(T_2)$.

6) Commutativity: If $T$ and $S$ are $DSH$ currents with complements bidegrees then
\begin{eqnarray*}
\lim _{n\rightarrow\infty}[\int _YK_n(T)\wedge S-\int _YT\wedge K_n(S)]=0.
\end{eqnarray*}

7) Compatibility with the differentials: $dd^cK_n^{\pm}(T)=K_n^{\pm}(dd^cT)$.

8) Condition on support: The support of $K_n(T)$ converges to the support of $T$. By this we mean that if $U$ is an open neighborhood of $supp(T)$, then  there is $n_0$ so that when $n\geq n_0$ then $supp(K_n(T))$ is contained in $U$. Moreover, the number $n_0$ can be chosen so that it depends only on $supp(T)$ and $U$ but not on the current $T$.

\label{DefinitionGoodApproximation}\end{definition}

Now we give the definition of pullback operator on $DSH^p(Y)$ currents
\begin{definition}
 Let $T$ be a $DSH^p(Y)$ current on $Y$. We say that $f^{\sharp}(T)$ is well-defined if there is a number $s\geq 0$ and a current $S$ on $X$ so that
\begin{eqnarray*}
\lim _{n\rightarrow\infty}f^*(K_n(T))= S,
\end{eqnarray*}
for any good approximation scheme by $C^{s+2}$ forms $K_n^{\pm}$. Then we write $f^{\sharp}(T)=S$.
 \label{DefinitionPullbackCurrentsByMeromorphicMaps}\end{definition}

By the commutativity property of good approximation schemes by $C^s$  forms, if $T$  is $DSH$ so that $f^{\sharp}(T)=S$ is well-defined then for any smooth form $\alpha$ we have
\begin{eqnarray*}
\int _Xf^{\sharp}(T)\wedge \alpha =\lim _{n\rightarrow\infty}\int _YT\wedge K_n(f_*(\alpha )).
\end{eqnarray*}
This equality helps to extend Definition \ref{DefinitionPullbackCurrentsByMeromorphicMaps} to any $(p,p)$ current $T$. Recall that since $Y$ is a compact manifold, any current on $Y$  is of finite order.
\begin{definition}
Let $T$ be a $(p,p)$ current of order $s_0$. We say that $f^{\sharp}(T)$ is well-defined if there is a number $s\geq s_0$ and a current $S$ on $X$ so
that
\begin{eqnarray*}
\lim _{n\rightarrow}\int _{Y}T\wedge K_n(f_*(\alpha )= \int _{X}S\wedge \alpha ,
\end{eqnarray*}
for any smooth form $\alpha$ on $X$ and any good approximation scheme by $C^{s+2}$ forms. Then we write
$f^{\sharp}(T)=S$. \label{DefinitionPullbackDdcOfOrderSCurrents}\end{definition}

\subsection{Results}

The operator $f^{\sharp}$ in Definitions \ref{DefinitionPullbackCurrentsByMeromorphicMaps} and \ref{DefinitionPullbackDdcOfOrderSCurrents} has the following properties:

\begin{lemma}
i) If $T$ is a continuous $(p,p)$ form (not necessarily $DSH$) then $f^{\sharp}(T)$ is well-defined and coincides with the standard definition
$f^{*}(T):=(\pi _1)_*(\pi _2^*(T)\wedge [\Gamma _f])$.

ii) $f^{\sharp}$ is closed under linear combinations: If $f^{\sharp}(T_1)$ and $f^{\sharp }(T_2)$ are well-defined, then so is
$f^{\sharp}(a_1T_1+a_2T_2)$ for any complex numbers $a_1$ and $a_2$. Moreover $f^{\sharp}(a_1T_1+a_2T_2)=a_1f^{\sharp}(T_1)+a_2f^{\sharp}(T_2)$.

iii) If $T$ is $DSH$ and $f^{\sharp}(T)$ is well-defined, then the support of $f^{\sharp}(T)$ is contained in $f^{-1}(supp(T))$.

iv) If $T$ is closed then $f^{\sharp}(T)$ is also closed, and in cohomology $\{f^{\sharp}(T)\}=f^*\{T\}$.
\label{LemmaGoodPropetiesOfPullbackOperator}\end{lemma}

For a smooth form, we can also define its pullback by using any desingularization of the graph of the map. We have an analog result
\begin{theorem}
Let $\widetilde{\Gamma _f}$ be a desingularization of $\Gamma _f$, and let $\pi :\widetilde{\Gamma _f}\rightarrow X$ and $g:\widetilde{\Gamma
_f}\rightarrow Y$ be the induced maps of $\pi _X|\Gamma _f$ and $\pi _Y|\Gamma _f$. Thus $\widetilde{\Gamma _f}$ is a compact K\"ahler manifold, $\pi $
is a modification, and $g$ is a surjective holomorphic map so that $f=g\circ \pi ^{-1}$. Let $T$ be a $(p,p)$ current on $Y$. If $g^{\sharp}(T)$ is
well-defined, then $f^{\sharp}(T)$ is also well-defined. Moreover $f^{\sharp}(T)=\pi _*(g^{\sharp}(T))$.
\label{TheoremPullbackByDesingularization}\end{theorem}

The following result is a restatement of a result of Dinh and Sibony (section 5 in \cite{dinh-sibony2}):
\begin{theorem}
Let $\theta $ be a smooth function on $X\times Y$ so that $supp(\theta )\cap \Gamma _f\subset \Gamma _f-\mathcal{C}_f$. Then for any $DSH^p$ current $T$ on $Y$,
$(\pi _X)_*(\theta [\Gamma _f]\wedge \pi _Y^{*}(T))$ is well-defined  (see also \cite{meo}).
\label{TheoremLocalPullbackForPositiveClosedCurrents}\end{theorem}

The following result is a generalization of a result proved by Dinh and Sibony in the case of projective spaces (see Proposition 5.2.4 in \cite{dinh-sibony4})
\begin{theorem}
Let $X$ and $Y$ be two compact K\"ahler manifolds. Let $f:X\rightarrow Y$ be a dominant meromorphic map. Assume that $\pi _X(\mathcal{C}_f)$ is of
codimension $\geq p$. Then $f^{\sharp}(T)$ is well-defined for any positive closed $(p,p)$ current $T$ on $Y$. Moreover the following continuity holds:  if $T_j$  are positive closed
$(p,p)$ currents weakly converging to $T$ then $f^{\sharp}(T_j)$ weakly converges to $f^{\sharp}(T)$. \label{TheoremInterestingExample1}\end{theorem}

{\bf Example 1}: In \cite{bedford-kim}, Bedford and Kim studied the linear quasi-automorphisms. These are birational selfmaps $f$ of rational
$3$-manifolds $X$ so that both $f$ and $f^{-1}$ have no exceptional hypersurfaces. Hence we can apply Theorem \ref{TheoremInterestingExample1} to
pullback and pushforward any positive closed $(2,2)$ current on $X$. The map $J_X$ in Section 4  is also a quasi-automorphism.

Below is a more general result, dealing with the case when the current $T$ is good (say continuous) outside a closed set $A$ whose preimage is not big.
\begin{theorem}
Let $X$ and $Y$ be two compact K\"ahler manifolds. Let $f:X\rightarrow Y$ be a dominant meromorphic map. Let $A\subset Y$ be a closed subset so that $f^{-1}(A)\cap \pi _X(\mathcal{C}_f)\subset V$ where $V$ is an analytic subvariety of $X$ having codim $\geq p$.  If $T$ is a positive closed $(p,p)$-current on
$Y$ which is continuous on $Y-A$, then $f^{\sharp}(T)$ is well-defined. Moreover, the following continuity holds: If $T_n^{\pm}$ are positive closed continuous $(p,p)$ forms so that $||T_n^{\pm}||$ are uniformly bounded, $T_n^+-T_n^{-}\rightharpoonup T$, and $T_n^{\pm}$ locally uniformly converges on $Y-A$, then $f^*(T_n^+-T_n^-)\rightharpoonup f^{\sharp}(T)$.
 \label{TheoremInterestingExample}\end{theorem}
When $\pi _X(\mathcal{C}_f)$ has codimension $\geq p$, then we can choose $A=Y$ in Theorem \ref{TheoremInterestingExample}, and thus recover Theorem \ref{TheoremInterestingExample1}.

As a consequence, we have the following result on pulling back of varieties:
\begin{corollary}
Let $f,X,Y$ be as in Theorem \ref{TheoremInterestingExample}. Let $V$ be an analytic variety of $Y$ of codim $p$. Assume that $f^{-1}(V)$ has codim $\geq p$. Then $f^{\sharp}[V]$ is well-defined.

\label{CorollaryPullBackVariety}\end{corollary} The assumptions in Corollary \ref{CorollaryPullBackVariety} are optimal, as can be seen from

{\bf Example 2}:  Let $Y=$ a compact K\"ahler $3$-fold, and let $L_0$ be an irreducible smooth curve in $Y$. Let $\pi :X\rightarrow Y$ be the blowup of
$Y$ along $L_0$. If $L$ is an irreducible curve in $Y$ which does not coincide with $L_0$ then $\pi ^{-1}(L)$ has dimension $1$, hence $\pi ^{\sharp}[L]$
is well-defined. In contrast, it is expected that $\pi ^{\sharp}[L_0]$ is  not well-defined. One explanation (which is communicated to us by Professor
Tien Cuong Dinh, see also the introduction in \cite{alessandrini-bassanelli2}) is that if $\pi ^{\sharp}[L_0]$ was to be defined, then it should be a
special $(2,2)$ current on the hypersurface $\pi ^{-1}(L_0)$. However, we have too many $(2,2)$ currents on that hypersurface to point out a special one.

We have the following example of losing positivity
\begin{corollary}
Let $X$ be the blowup of $\mathbb{P}^3$ along $4$ points $e_0=[1:0:0:0],e_1=[0:1:0:0],e_2=[0:0:1:0], e_3=[0:0:0:1]$. Let $J:\mathbb{P}^3\rightarrow
\mathbb{P}^3$ be the Cremona map $J[x_0:x_1:x_2:x_3]=[1/x_0:1/x_1:1/x_2:1/x_3]$, and let $J_X$ be the lifting of $J$ to $X$.

For $0\leq i\not= j\leq 3$, let $\Sigma _{i,j}$ be the line in $\mathbb{P}^3$ consisting of points $[x_0:x_1:x_2:x_3]$ where $x_i=x_j=0$. Let
$\widetilde{\Sigma _{i,j}}$ be the strict transform of $\Sigma _{i,j}$ in $X$.

For any positive closed $(2,2)$ current $T$, $J_X^{\sharp}(T)$ is well-defined. Moreover, $J_X^{\sharp}([\widetilde{\Sigma _{0,1}}])= -
[\widetilde{\Sigma _{2,3}}]$ and $J_X^{\sharp}([\widetilde{\Sigma _{2,3}}])= - [\widetilde{\Sigma _{0,1}}]$. \label{CorollaryTheMapJ}\end{corollary}

\begin{remark}
The map $J_X$ was given in Example 2.5 page 33 in \cite{guedj} where the author showed that the map $J_X^*:H^{2,2}(X)\rightarrow H^{2,2}(X)$ does not
preserve the cone of cohomology classes generated by positive closed $(2,2)$ currents.

In Lemma \ref{LemmaJXSquare}, it will be shown that $(J_X^{\sharp})^2(T)=T$ for any positive closed $(2,2)$ current $T$. Thus this example gives positive support to an open question
posed in Section 6.
\end{remark}

We conclude this subsection discussing pullback of a positive closed $(p,p)$ current $T$ in general. For $c>0$ define $E_c(T)=\{y\in Y: \nu (T,y)\geq
c\}$, where $\nu (T,y)$ is the Lelong number of $T$ at $y$ (see \cite{demailly} for definition). Then by the semi-continuity theorem of Siu (see
\cite{siu}, and also \cite{demailly}), $E_c(T)$ is an analytic subvariety of $Y$ of codimension $\geq p$. Moreover, we have a decomposition
\begin{eqnarray*}
T=R+\sum _{j=1}^{\infty}\lambda _j[V_j],
\end{eqnarray*}
where $\lambda _j\geq 0$, $V_j$ is an irreducible analytic variety of codimension $p$ and is contained in $E(T)=\cup _{c>0}E_c(T)$, and $R$ is a positive
closed current such that $E_c(R)$ has codimension $>p$ for all $c>0$. Note that $E(T)=$ the union of $E_c(T)$'s for rational numbers $c>0$, hence is a
(at most) countable union of analytic varieties.

\begin{theorem}
Notations are as above. Assume that for any irreducible variety $V$ of codimension $p$ contained in $E(T)$, then $f^{-1}(V)$ has codimension $\geq p$.
Then $f^{\sharp}(\sum _{j=1}^{\infty}\lambda _j[V_j])$ is well-defined and is equal to $\sum _{j=1}^{\infty}\lambda _jf^{\sharp}[V_j]$. Hence
$f^{\sharp}(T)$ is well-defined iff $f^{\sharp}(R)$ is well-defined. \label{TheoremUseSiuDecomposition}\end{theorem}

\subsection{Compatibility with previous works}

In this subsection we compare our results with the results in previous papers.

The pullback of positive closed $(1,1)$ currents was defined by Meo \cite{meo} for finite holomorphic maps between complex manifolds (not necessarily
compact or K\"ahler). Our definition coincides with his in the case of compact K\"ahler manifolds
\begin{corollary}
Let $X$ and $Y$ be two compact K\"ahler manifolds. Let $f:X\rightarrow Y$ be a dominant meromorphic map. Let $T$ be a positive closed $(1,1)$-current on
$Y$. Then $f^{\sharp}(T)$ is well-defined, and coincides with the usual definition. \label{CorollaryPositiveClosed11Currents}\end{corollary}
\begin{proof}
Since $\pi _X(\mathcal{C}_f)$ is a proper analytic subvariety of $X$, it has codimension $\geq 1$, thus we can apply Theorem
\ref{TheoremInterestingExample1}.
\end{proof}

The pullback of positive $dd^c$ closed $(1,1)$ currents were defined by Alessandrini - Bassanel \cite{alessandrini-bassanelli2} and Dinh -Sibony
\cite{dinh-sibony2} under several contexts. Our definition coincides with theirs in the case of compact K\"ahler manifolds
\begin{corollary}
Let $X$ and $Y$ be two compact K\"ahler manifolds. Let $f:X\rightarrow Y$ be a dominant meromorphic map. Let $T$ be a positive $dd^c$- closed $(1,1)$-current on
$Y$. Then $f^{\sharp}(T)$ is well-defined, and coincides with the usual definition. \label{CorollaryPositiveDD^cClosed11Currents}\end{corollary}
\begin{proof}
Consider a desingulariztion $\widetilde{\Gamma _f}$ and $\pi :\widetilde{\Gamma _f}\rightarrow X$ and $g:\widetilde{\Gamma _f}\rightarrow Y$ as in
Theorem \ref{TheoremPullbackByDesingularization}. Then it suffices to show that $g ^{\sharp}(T)$ is well-defined.  This later follows from the proof of
Theorem 5.5 in \cite{dinh-sibony2}.
\end{proof}

For a map $f:\mathbb{P}^k\rightarrow \mathbb{P}^k$, Russakovskii and Shiffman \cite{russakovskii-shiffman} defined the pullback of a linear subspace $V$
of codimension $p$ in $\mathbb{P}^k$ for which $\pi _2^{-1}(V)\cap \Gamma _f$ has codimension $\geq p$ in $\Gamma _f$. It can be easily seen that this is
a special case of Corollary \ref{CorollaryPullBackVariety}. In the same paper, we also find a definition for pullback of a measure having no mass on $\pi
_Y(\mathcal{C}_f)$. Our definition coincides with theirs

\begin{theorem} Let $X$ and $Y$ be two compact K\"ahler manifolds. Let $f:X\rightarrow Y$ be a dominant meromorphic map. Let $T$ be a positive measure having no
mass on $\pi _Y(\mathcal{C}_f)$. Then $f^{\sharp}(T)$ is well-defined, and coincides with the usual definition. Moreover, if $T$ has no mass on proper analytic
subvarieties of $Y$, then $f^{\sharp}(T)$ has no mass on proper analytic subvarieties
of $X$.

\label{TheoremPullbackOfMeasures}\end{theorem}

\subsection{Applications}\label{SubsectionApplications}

We now discuss the problem of finding an invariant current of a dominant meromorphic self-map $f$. Let $f:X\rightarrow X$ be a dominant meromorphic
selfmap of a compact K\"ahler manifold $X$ of dimension $k$. Define by $r_p(f)$ the spectral radius of $f^*:H^{p,p}(X)\rightarrow H^{p,p}(X)$. Then the
$p$-th dynamical degree of $f$ is defined as follows:
\begin{eqnarray*}
\delta _p(f)=\lim _{n\rightarrow\infty}(r_p(f^n))^{1/n},
\end{eqnarray*}
where $f^n=f\circ f\circ \ldots \circ f$ is the $n$-th iteration of $f$. When $p=dim(X)$ then $\delta _p(f)$ is the topological degree of $f$.

The map $f$ is called $p$-algebraic stable (see, for example \cite{dinh-sibony4}) if $(f^*)^n=(f^n)^*$ as linear maps on $H^{p,p}(X)$ for all
$n=1,2,\ldots $. When this condition is satisfied, it follows that $\delta _p(f)=r_p(f)$, thus helps in determining the $p$-th dynamical degree of $f$.

There is also the related condition of $p$-analytic stable (see \cite{dinh-sibony4}) which requires that

1) $(f^{n})^{\sharp}(T)$ is well-defined for any positive closed $(p,p)$ current $T$ and any $n\geq 1$.

2) Moreover, $(f^n)^{\sharp}(T)=(f^{\sharp})^n(T)$ for any positive closed $(p,p)$ current $T$ and any $n\geq 2$.

Since $H^{p,p}(X)$ is generated by classes of positive closed smooth $(p,p)$ forms, $p$-analytic stability implies $p$-algebraic stability. In fact, if $\pi _1(\mathcal{C}_f)$ has codimension $\geq p$, then $f$ is $p$-analytic stable iff it is $p$-algebraic stable and satisfies condition 1) above so that $(f^{\sharp})^n(\alpha )$ is positive closed for
any positive closed smooth $(p,p)$ form and for any $n\geq 1$. Hence $1$-algebraic stability is the same as $1$-analytic stability.

For any map $f$ then $f$ is $k$-algebraic stable where $k=$dimension of $X$. If $f$ is holomorphic then it is $p$-algebraic stable for any $p$. We have
the following result

\begin{lemma}
Let $X$ be a compact Kahler manifold with a Kahler form $\omega _X$ and $f:X\rightarrow X$ be a dominant meromorphic map. Assume that $\pi
_1(\mathcal{C}_f)$ has codimension $\geq p$ and $f$ is $p$-analytic stable. Let $0\not= \theta$ be an eigenvector with respect to the eigenvalue $\lambda
=r_p(f)$ the spectral radius of the linear map $f^*:H^{p,p}(X)\rightarrow H^{p,p}(X)$. Assume moreover that $||(f^n)^*(\omega _X^p)||\thicksim \lambda
^n$ as $n\rightarrow\infty$. Then there is a closed $(p,p)$ current $T$ which is a difference of two positive closed $(p,p)$ currents satisfying
$\{T\}=\theta$ and $f^{\sharp}(T)=\lambda T$. \label{LemmaInvariantCurrentForKahlerClass}\end{lemma}

Since $f$ is $p$-analytic stable, the condition on $||(f^n)^*(\omega _X^p)||$ can be easily checked by look at the Jordan form for $f^*$ (see e.g.
\cite{guedj}). Variants of this condition are also available. Lemma \ref{LemmaInvariantCurrentForKahlerClass} generalizes the results for the standard
case $p=1$ and for the case $X=\mathbb{P}^k$ in Dinh and Sibony \cite{dinh-sibony4}. We suspect that the pseudo-automorphism in \cite{bedford-kim} are
$2$-analytic stable, the latter may probably be checked using the method of the proof of Lemma \ref{LemmaJXSquare}. If so, Lemma
\ref{LemmaInvariantCurrentForKahlerClass} can be applied to these maps to produce invariant closed $(2,2)$ currents. However, these invariant currents
may not be unique, since for the maps in \cite{bedford-kim} the first and second dynamical degrees are the same. The map $J_X$ in Section 4 has invariant
$(2,2)$ current $\widetilde{\Sigma _{0,1}}-\widetilde{\Sigma _{2,3}}$ which is not positive. The relation between $p$-algebraic and $p$-analytic
stabilities to the problem of finding invariant currents will be discussed more in Sections $5$ and $6$.

Let us continue with an application concerning invariant positive closed currents whose supports are contained in pluripolar sets.
\begin{corollary}
Let $f_1:\mathbb{P}^{k_1}\rightarrow \mathbb{P}^{k_1}$ and $f_2:\mathbb{P}^{k_2}\rightarrow \mathbb{P}^{k_2}$ be dominant rational maps not $1$-algebraic
stable, of degrees $d_1$ and $d_2$ respectively. Then there is a nonzero positive closed $(2,2)$ current $T$ on $\mathbb{P}^{k_1}\times \mathbb{P}^{k_2}$
with the following properties:

1) $f^{\sharp}(T)$ is well-defined and moreover $f^{\sharp}(T)=d_1d_2T$, here $f=f_1\times f_2$.

2) The support of $T$ is pluripolar.
 \label{CorollaryInvariantCurrentsForNonAlgebraicStability}\end{corollary}
The existence of Green currents $T_1$ and $T_2$ for $f_1$ and $f_2$ were proved by Sibony \cite{sibony} (see also \cite{buff}). The current $T$ is in
fact the product $T_1\times T_2$. Its support is contained in a countable union of analytic varieties of codimension $2$ in $\mathbb{P}^{k_1}\times
\mathbb{P}^{k_2}$. The subtlety in proving the Conclusion 1) of Corollary \ref{CorollaryInvariantCurrentsForNonAlgebraicStability} lies in the fact that
for general choices of $f_1$ and $f_2$ it is not clear that we can pullback every positive closed $(2,2)$ currents, and even if we can do so, we may not
have the continuity on pullback like in the case of $(1,1)$ currents.

\begin{corollary}
Let $X$ be a compact K\"ahler manifold of dimension $k$, and let $f:X\rightarrow X$ be a dominant meromorphic map. Assume that $f$ has large topological
degree, i.e. $\delta _k(f)>\delta _{k-1}(f)$. Then $f$ has an invariant positive measure $\mu$, i.e. $f^*(\mu )=\delta _k(f)\mu$.
\label{TheoremInvariantMeasureForLargeTopologicalDegree}\end{corollary} The result of Corollary \ref{TheoremInvariantMeasureForLargeTopologicalDegree}
belongs to Guedj \cite{guedj1}. Our proof here is slightly different from his proof in that we don't need to show that the measure $\mu$ has no mass on
proper analytic subvarieties.

\begin{corollary}
Let $X$ be a compact K\"ahler manifold, and let $f:X\rightarrow X$ be a surjective holomorphic map. Let $\lambda$ be a real eigenvalue of
$f^*:H^{p,p}(X)\rightarrow H^{p,p}(X)$, and let $0\not= \theta _{\lambda }\in H^{p,p}(X)$ be an eigenvector with eigenvalue $\lambda$. Assume moreover
that $|\lambda |>\delta _{p-1}(f)$. Then there is a closed current $T$ of order $2$ with $\{T\}=\theta _{\lambda }$ so that $f^{\sharp}(T)$ is
well-defined, and moreover $f^{\sharp}(T)=\lambda T$. \label{TheoremInvariantCurrentsForHyperpolicCohomologyHolomorphicMaps}\end{corollary}

{\bf Example 3:} Let $X=\mathbb{P}^2_{w_1}\times \mathbb{P}^2_{w_2}\times \mathbb{P}^2_{w_3}$, and let $f:X\rightarrow X$ to be
$f(w_1,w_2,w_3)=(P_2(w_2),P_3(w_3),P_1(w_1))$ where $P_1,P_2,P_3:\mathbb{P}^2\rightarrow \mathbb{P}^2$ are surjective holomorphic maps of degrees $\geq
2$, and not all of them are submersions (For example, we can choose one of them to be $P[z_0:z_1:z_2]=[z_0^d:z_1^d:z_2^d]$ for some integer $d\geq 2$).
Theorem \ref{TheoremInvariantCurrentsForHyperpolicCohomologyHolomorphicMaps} can be applied to find invariant currents for $f$.

\subsection{Acknowledgments and organization of the paper}

The author would like to thank his advisor Professor Eric Bedford for continuing guidance, support and encouragement in his studying in Indiana
University Bloomington. The author is grateful to Professor Tien Cuong Dinh for many explanations and discussions on the subjects of pullback of
currents, regularization of currents, and super-potentials; to Professor Jean Pierre Demailly for many explanations on positive closed currents; and to
Professor Phuc Cong Nguyen for many discussions on the integral kernels. The author is thankful to Professor Nessim Sibony for insightful suggestions on
earlier versions of the paper and for stimulating discussions which helped to improve the results and the exposition of the paper;  and to Professor
Lucia Alessandrini for helping with the Federer type support theorem in \cite{bassanelli} and for encouragement on an earlier version of this paper. The
discussions with Professors Jeffrey Diller and Norman Levenberg, and with the author's colleagues Turgay Bayraktar, Janli Lin and Thang Quang Nguyen, are
also invaluable.

Part of the paper is written while the author is visiting University of Paris 11. He would like to thank the university for its support and hospitality.

The rest of this paper is organized as follows: In Section 2 we collect some simple but helpful properties of positive currents. Then we consider the
pull-back operator in Section 3. In Section 4 we explore the properties of the map $J_X$. We will also give results concerning the operator $f^o$ on
positive closed currents defined by Dinh-Nguyen \cite{dinh-nguyen} (see Proposition \ref{PropositionTheOperatorDotPullback}), and concerning the
regularization results of Dinh-Sibony \cite{dinh-sibony1} (see Proposition \ref{PropositionNoVeryGoodApproximation}). In Section 5 we consider invariant
currents. We give examples of good approximation schemes and discuss some open questions in the last section.

\section{Some preliminary results}
In this section, we collect some simple but useful facts about positive currents. All the results presented are well known, but we include the proofs for the convenience of the readers. Through out this section, let $Z$ be a compact K\"ahler manifold of dimension $k$, with a K\"ahler $(1,1)$ form $\omega _Z$. Let $\pi _1,\pi _2:Z\times Z\rightarrow Z$ be the projections, and let $\Delta _Z\subset Z\times Z$ be the diagonal.

\begin{lemma}
Let $T$ be a continuous real $(p,p)$ form on $Z$. Then there exists a constant $A>0$ independent of $T$ so that
\begin{eqnarray*}
A||T || _{L^{\infty}}\omega _Z^p\pm T
\end{eqnarray*}
are both strongly positive forms.
\label{LemmaBoundContinuousForms}\end{lemma}
\begin{proof}
Since $Z$ is a compact K\"ahler manifold, there is a finite covering of $Z$ by open sets $U$'s each of them is biholomorphic to a ball in $\mathbb{C}^k$.
Using a partition of unity for this covering, we reduce the problem to the case where $T$ is a continuous real $(p,p)$ form compactly supported in a ball
in $\mathbb{C}^k$. Since $T$ is a real form, we can write
\begin{eqnarray*}
T=\sum _{|I|=|J|=p}(f_{I,J}dz_I\wedge d\overline{z_J}+\overline{f_{I,J}}d\overline{z_I}\wedge dz_J),
\end{eqnarray*}
where $f_{I,J}$ are bounded continuous complex-valued functions. By Lemma 1.4 page 130 in \cite{demailly}, $dz_I\wedge d\overline{z_J}$ can be
represented as a linear combination of strongly positive forms with complex coefficients. Let us write
\begin{eqnarray*}
dz_I\wedge d\overline{z_J}=\sum _{i\in \mathcal{A}}\alpha _{I,J,i}\varphi _{i},
\end{eqnarray*}
where $\mathcal{A}$ is a finite set independent of $I$ and $J$, $\varphi _i$ are fixed strongly positive $(p,p)$ forms, and $\alpha _{I,J,i}$ are complex
numbers. Then
\begin{eqnarray*}
d\overline{z_I}\wedge d{z_J}=\sum _{i\in \mathcal{A}}\overline{\alpha _{I,J,i}}\varphi _{i}.
\end{eqnarray*}
Hence $T$ can be represented in the form
\begin{eqnarray*}
T=\sum _{|I|=|J|=p}\sum _{i\in \mathcal{A}}f_{I,J,i}\varphi _{i},
\end{eqnarray*}
where $f_{I,J,i}=\alpha _{I,J,i}f_{I,J}+\overline{\alpha _{I,J,i}f_{I,J}}$ are bounded continuous real-valued functions satisfying $||f_{I,J,i}||_{L^{\infty}}\leq
A||T||_{L^{\infty}}$ for some constant $A>0$ independent of $T$. Each of the forms $\varphi _i$ can be bound by a multiplicity of $\omega _Z^p$, hence we
can find a constant $A>0$ independent of $T$ so that $A||T || _{L^{\infty}}\omega _Z^p\pm T$ are strongly positive forms.
\end{proof}
\begin{lemma}
Let $S$ be a strongly positive current on $Z$, and let $T$ be a continuous positive $(p,p)$ form. Then $S\wedge T$ is well-defined and is a positive
current.

Similarly, if $S$ is a positive current on $Z$, and $T$ is a continuous strongly positive $(p,p)$ form then $S\wedge T$ is well-defined and is a positive
current. \label{LemmaIntersectionOfPositiveContinuousForms}\end{lemma}
\begin{proof}
Since $S$ is a strongly positive current on $Z$, it is of order zero, hence can be wedged with a continuous form. Thus $S\wedge T$ is well-defined. Now
we show that $S\wedge T$ is a positive current.

We can approximate $T$ uniformly by smooth $(p,p)$ forms $T_n$. Then use Lemma \ref{LemmaBoundContinuousForms}, there is a constant $A>0$ independent of
$n$ so that $A||T-T_n||_{L^{\infty}}\omega _Z^p\pm (T-T_n)$ are strongly positive. Since $T$ is a positive form, this implies that
$T_n+A||T-T_n||_{L^{\infty}}\omega _Z^p$ are positive for all $n$. Since the current $S$ acts continuously on $C^0$ forms, and we chose $T_n$ to converge
uniformly to $T$, we have that
\begin{eqnarray*}
S\wedge T=\lim _{n\rightarrow\infty}S\wedge T_n=\lim _{n\rightarrow\infty}S\wedge (T_n +A||T-T_n||_{L^{\infty}}\omega _Z^p).
\end{eqnarray*}
Since $S$ is strongly positive and $T_n +A||T-T_n||_{L^{\infty}}\omega _Z^p$ are positive smooth forms, $S\wedge (T_n +A||T-T_n||_{L^{\infty}}\omega
_Z^p)$ are positive currents. Thus $S\wedge T$ is the weak limit of a sequence of positive currents, hence itself a positive current.
\end{proof}
\begin{lemma}
Let $T$ be a positive closed $(p,p)$ current on $Z$. Then there is a closed smooth $(p,p)$ form $\theta$ on $Z$ so that $\{\theta\}=\{T\}$ in cohomology,
and moreover
\begin{eqnarray*}
-A ||T|| \omega _Z^p\leq \theta \leq A ||T||\omega _Z^p.
\end{eqnarray*}
Here $A>0$ is independent of $T$. \label{LemmaSmoothRepresentativeOfPositiveClosedCurrents}\end{lemma}
\begin{proof}
Let $\pi _1,\pi _2:Z\times Z\rightarrow Z$ be the two projections, and let $\Delta _Z$ be the diagonal of $Z$. Let $\Delta$ be a closed smooth form on
$Z\times Z$ representing the cohomology class of $[\Delta _Z]$. If we define
\begin{eqnarray*}
\theta =(\pi _1)_*(\pi _2^*(T)\wedge \Delta ),
\end{eqnarray*}
it is a smooth $(p,p)$ current on $Z$ having the same cohomology class as $T$. Since $Z$ is compact, so is $Z\times Z$, and by Lemma
\ref{LemmaBoundContinuousForms} there is a constant $A>0$ so that $A (\pi _1^*\omega _Z+\pi _2^*\omega _Z)^{dim (Z)}\pm\Delta$ are strongly positive
forms. Since $T$ is a positive current, by Lemma \ref{LemmaIntersectionOfPositiveContinuousForms} it follows that
\begin{eqnarray*}
\theta =(\pi _1)_*(\pi _2^*(T)\wedge \Delta )\leq A(\pi _1)_*((\pi _1^*\omega _Z+\pi _2^*\omega _Z)^{dim (Z)}\wedge \pi _2^*(T))=A||T||\omega _Z^p.
\end{eqnarray*}
Similarly, we have also $\theta \geq -A ||T|| \omega _Z^p$.
\end{proof}
\begin{lemma}
Let $T_j$ be a sequence of $DSH^p(Z)$ currents converging in $DSH$ to a current $T$. Then for any continuous $(k-p,k-p)$ form $S$ we have
\begin{eqnarray*}
\lim _{j\rightarrow \infty}\int _ZT_j\wedge S=\int _ZT\wedge S.
\end{eqnarray*}
\label{LemmaConvergenceOfDSHCurrents}\end{lemma}
\begin{proof}
By assumption, $T_j$ weakly converges to $T$ in the sense of currents, and moreover we can write $T_j=T_j^{+}-T_j^{-}$ and $T=T^{+}-T^-$ where
$T_j^{\pm}$ and $T^{\pm}$ are positive currents, whose norms are uniformly bounded. Since $S$ is a continuous form, we can find a sequence of smooth
forms $S_n$ uniformly converging to $S$, i.e. we can choose $S_n$ smooth forms so that
\begin{eqnarray*}
-\frac{1}{n}\omega _Z^{k-p}\leq S-S_n\leq \frac{1}{n}\omega _Z^{k-p}.
\end{eqnarray*}
Hence by Lemma \ref{LemmaIntersectionOfPositiveContinuousForms}, for any $j$ and $n$
\begin{eqnarray*}
-\frac{1}{n}(||T_j^+||+||T_j^-||)+\int _ZT_j\wedge S_n \leq \int _ZT_j\wedge S\leq \frac{1}{n}(||T_j^+||+||T_j^-||)+\int _ZT_j\wedge S_n.
\end{eqnarray*}
Hence given a number $n$, letting $j\rightarrow \infty$, using the fact that $T_j\rightharpoonup T$, $S_n$ is smooth, and $||T_j||_{DSH}$ is uniformly
bounded
\begin{eqnarray*}
-\frac{A}{n}+\int _ZT\wedge S_n \leq \liminf _{j\rightarrow\infty}\int _ZT_j\wedge S\leq \limsup _{j\rightarrow\infty}\int _ZT_j\wedge S\leq
\frac{A}{n}+\int _ZT\wedge S_n,
\end{eqnarray*}
where $A>0$ is independent of $n$. Since $T$ is a difference of two positive currents, it is a current of order zero, hence acting continuously on the
space of continuous forms equipped with the sup norm. Since $S_n$ converges uniformly to $S$, we have
\begin{eqnarray*}
\lim _{n\rightarrow\infty}\int _ZT\wedge S_n =\int _ZT\wedge S.
\end{eqnarray*}
Combining this and the previous inequalities, letting $n\rightarrow \infty$, we obtain
\begin{eqnarray*}
\lim _{j\rightarrow \infty}\int _ZT_j\wedge S=\int _ZT\wedge S,
\end{eqnarray*}
as wanted.
\end{proof}

\section{Pull-back of $DSH$ currents}

First, we show the good properties of the operator $f^{\sharp}$
\begin{proof}(Of Lemma \ref{LemmaGoodPropetiesOfPullbackOperator}) Let $K_n=K_n^+-K_n^-$ be a good approximation scheme by $C^2$ forms.

i)  If $T$ is a continuous form, then $K_n^{\pm}(T)$ uniformly converges on $Y$. Hence there are continuous forms $T^+,T^-$ and constants $\epsilon _n$ decreasing to $0$, so that $T=T^+-T^-$ and $-\epsilon _n\omega _Y^p\leq K_n^{\pm}(T)-T^{\pm}\leq \epsilon _n\omega _Y^p$. Then
\begin{eqnarray*}
-\epsilon _nf^*(\omega _Y^p)\leq f^*(K_n^{\pm}(T))-f^*(T^{\pm})\leq \epsilon _nf^*(\omega _Y^p),
\end{eqnarray*}
and thus $f^*(K_n^{\pm}(T))$ weakly converges to $f^*(T^{\pm})$.  Therefore, $f^*(K_n^{+}(T)-K_n^{-}(T))$ weakly converges to $f^*(T^+)-f^*(T^-)=f^*(T)$. This shows that  $f^{\sharp}(T)$ is well-defined and coincides with the usual definition.

ii) Follows easily from the definition.

iii) If $T$ is $DSH$, the result follows from the definition and the fact that support of $K_n(T)$ converges to support of $T$.

iv) First we show that if $T=T_1+dd^cT_2$ is closed, where $T_1$ is a $(p,p)$ current and $T_2$ is a $(p-1,p-1)$ current both of order $0$, and $f^{\sharp}(T)$ is well-defined, then $f^{\sharp}(T)$ is closed.

From the assumption, it follows that $T_1$ is closed. To show that $f^{\sharp}(T)$ is closed, it suffices to show that if $\alpha$ is a $d$-exact $(dim(X)-p, dim(X)-p)$ smooth form, then
\begin{eqnarray*}
\int _Xf^{\sharp}(T)\wedge \alpha =0.
\end{eqnarray*}
In fact, by definition
\begin{eqnarray*}
\int _Xf^{\sharp}(T)\wedge \alpha =\lim _{n\rightarrow\infty}\int _YT_1\wedge K_n(f_*(\alpha ))+T_2\wedge dd^cK_n(f_*(\alpha )).
\end{eqnarray*}
By the $dd^c$ lemma, there is a smooth form $\beta$ so that $\alpha =dd^c(\beta )$. Then by the compatibility with differentials of good approximation schemes, we have $K_n(f_*(\alpha ))=K_n(f_*(dd^c\beta ))=dd^cK_n(f_*(\beta ))$ is $d$-exact. Thus each of the two integrals in the RHS of the above equality is $0$, independent of $n$. Hence the limit is $0$ as well.

Now we show that $\{f^{\sharp}(T)\}=f^*\{T\}$. Let $\theta$ be a smooth closed form so that $\{T\}=\{\theta\}$. Then there is a current $R$ so that $T-\theta =dd^c(R)$. If $\alpha $ is a closed smooth form then
\begin{eqnarray*}
\int _X(f^{\sharp }(T)-f^*(\theta ))\wedge \alpha &=&\lim _{n\rightarrow\infty}\int _Y(T-\theta )\wedge K_n(f_*(\alpha ))\\
&=&\lim _{n\rightarrow\infty}\int _Ydd^c(R)\wedge K_n(f_*(\alpha ))\\
&=&\lim _{n\rightarrow\infty}\int _YR\wedge K_n(f_*(dd^c \alpha ))=0,
\end{eqnarray*}
since $dd^c(\alpha )=0$. This shows that $\{f^{\sharp}(T)\}=\{f^*(\theta )\}$, and the latter is $f^*\{T\}$ by definition.

\end{proof}

\begin{proof} (Of Theorem \ref{TheoremPullbackByDesingularization}) Assume that $g^{\sharp}(T)$ is well-defined with respect to number $s$ in Definition
\ref{DefinitionPullbackDdcOfOrderSCurrents}. Let $\alpha$ be a smooth form on $X$ and $K_n$ a good approximation scheme by $C^{s+2}$ forms on $Y$. Then
$f_{*}(\alpha )=g_{*}(\pi ^{*}\alpha )$. Since $\pi ^{*}(\alpha )$ is smooth on $\widetilde{\Gamma _f}$ and $g^{\sharp}(T)$ is well-defined, we have
\begin{eqnarray*}
\lim _{n\rightarrow\infty}\int _YT\wedge K_n(f_{*}\alpha )&=&\lim _{n\rightarrow\infty}\int _YT\wedge K_n(g_{*}\pi ^*\alpha )\\
&=&\int _{\widetilde{\Gamma _f}}g^{\sharp}(T)\wedge \pi ^*\alpha =\int _{X}\pi _*g^{\sharp}(T)\wedge \alpha ,
\end{eqnarray*}
as wanted.
\end{proof}

Now we give the proofs of Theorems \ref{TheoremLocalPullbackForPositiveClosedCurrents}, \ref{TheoremInterestingExample1},
\ref{TheoremInterestingExample}, \ref{TheoremUseSiuDecomposition} and \ref{TheoremPullbackOfMeasures}.

\begin{proof} (Of Theorem \ref{TheoremLocalPullbackForPositiveClosedCurrents})
In this proof we use the value $s=0$ in Definitions \ref{DefinitionPullbackCurrentsByMeromorphicMaps} and \ref{DefinitionPullbackDdcOfOrderSCurrents}.
The proof is the same as the proof of Lemma 3.3 in \cite{dinh-sibony2} using the following observations:

i) Lemma 3.1 in \cite{dinh-sibony2} applies for $C^2$ forms $T_n$. Hence Lemma 3.3 in \cite{dinh-sibony2} applies to $C^2$ forms $T_n$.

ii) Let us choose two difference good approximation schemes by $C^2$ forms $K_n=K_n^+-K_n^-$ and $H_n=H_n^+-H_n^-$. Then the sequences $K_n^+(T)+H_n^-(T)$ and
$K_n^-(T)+H_n^+(T)$ converges in $DSH$ to a same positive current.

iii) Apply Lemma 3.3 in \cite{dinh-sibony2} to the sequences $K_n^+(T)+H_n^-(T)$ and $K_n^-(T)+H_n^+(T)$, we conclude that in $\Gamma
_f-\mathcal{C}_f$, the sequences $f^*(K_n^+(T))+f^*(H_n^-(T))$ and $f^*(K_n^-(T))+f^*(H_n^+(T))$ converges to a same current. Thus we have that the sequences $f^*(K
_n^+(T)-K_n^-(T))$ and $f^*(H_n^+(T)-H_n^-(T))$ converges in $\Gamma _f-\mathcal{C}_f$ to a same current.
\end{proof}

\begin{proof} (Of Theorem \ref{TheoremInterestingExample1}) We follow the proof of Proposition 5.2.4 in \cite{dinh-sibony4} with some appropriate modifications. Let $K_n=K_n^+-K_n^-$ be a good approximation scheme by $C^2$ forms.

a) First we show that $f^{\sharp}(T)$ is well-defined for any positive closed $(p,p)$ current $T$.

Let $\theta$ be a smooth closed $(p,p)$ form so that $\{\theta\}=\{T\}$ in cohomology classes. Since $T=(T-\theta )+\theta $, by Lemma
\ref{LemmaGoodPropetiesOfPullbackOperator}, to show that $f^{\sharp}(T)$ is well-defined, it is enough to show that $f^{\sharp}(T-\theta )$ is
well-defined. By $dd^c$ lemma (see also \cite{dinh-sibony5}), there is a $DSH$ current $R$ so that $T-\theta =dd^c(R)$. Hence to show that
$f^{\sharp}(T-\theta )$ is well-defined, it is enough to show that $f^{\sharp}(R)$ is well-defined.

We can write $K_{n}(R)=R_{1,n}-R_{2,n}$, where $R_{i,n}$ are positive $(p-1,p-1)$ forms of class
$C^2$, and $dd^c(R_{i,n})=\Omega _{i,n}^+-\Omega _{i,n}^-$, where $\Omega _{i,n}^{\pm}$ are positive closed $C^2$ $(p,p)$ forms.
Moreover, $||R_{i,n}||$ and $||\Omega _{i,n}^{\pm}||$ are uniformly bounded.

i) First we show that $||f^*(R_{i,n})||$ are uniformly bounded. Theorem
\ref{TheoremLocalPullbackForPositiveClosedCurrents} implies that $f^*(R_{i,n})$ converges in $X-\pi _X(\mathcal{C}_f)$ to a current. Since the codimension of $\pi _X(\mathcal{C}_f)$ is $\geq p$, it is weakly $p$-pseudoconvex (see Lemma 5.2.2 in \cite{dinh-sibony4}). Hence there exists
a smooth $(dim (X)-p,dim (X)-p)$ form $\Theta $ defined on $X$ so that $dd^c\Theta \geq 2\omega _X^{dim(X)-p+1}$ on $\pi _X(\mathcal{C}_f)$. We can
choose a small neighborhood $V$ of $\pi _X(\mathcal{C}_f)$ so that $dd^c\Theta \geq \omega _X^{dim(X)-p+1}$ on $V$. Since $R_{i,n}$ is a positive
$C^2$ form, $f^*(R_{i,n})$ is well defined and is a positive current. Since $f^*(R_{i,n})$ converges in $X-\pi _X(\mathcal{C}_f)$ to a current,
it follows that $||f^*(R_{i,n})||_{X-V}$ is bounded. Because
\begin{eqnarray*}
||f^*(R_{i,n})||_X=||f^*(R_{i,n})||_{X-V}+||f^*(R_{i,n})||_{V},
\end{eqnarray*}
to show that $||f^*(R_{i,n})||_{X}$ is bounded, it is enough to estimate $||f^*(R_{i,n})||_V$. We have
\begin{eqnarray*}
||f^*(R_{i,n})||_V&=&\int _Vf^*(R_{i,n})\wedge \omega _X^{dim(X)-p+1}\leq \int _Vf^*(R_{i,n})\wedge dd^c(\Theta )\\
&=&\int _Xf^*(R_{i,n})\wedge dd^c(\Theta )-\int _{X-V}f^*(R_{i,n})\wedge dd^c(\Theta ).
\end{eqnarray*}
The term
\begin{eqnarray*}
|\int _{X-V}f^*(R_{i,n})\wedge dd^c(\Theta )|
\end{eqnarray*}
can be bound by $||f^*(R_{i,n})||_{X-V}$, and thus is bounded. We estimate the other term: Since $X$ is compact
\begin{eqnarray*}
|\int _Xf^*(R_{i,n})\wedge dd^c(\Theta )|&=&|\int _Xdd^cf^*(R_{i,n})\wedge \Theta |=|\int _Xf^*(dd^cR_{i,n})\wedge \Theta |\\
&=&|\int _Xf^*(\Omega _{i,n}^+-\Omega _{i,n}^-)\wedge \Theta |.
\end{eqnarray*}
Since $\Omega _{i,n}^{\pm}$ are positive closed $C^2$ forms, $f^*(\Omega _{i,n}^{\pm})$ are well-defined and are positive closed currents.
Choose a constant $A>0$ so that $A\omega _X^{dim(X)-p}\pm \Theta$ are strictly positive forms, we have
\begin{eqnarray*}
&&|\int _Xf^*(\Omega _{i,n}^+-\Omega _{i,n}^-)\wedge \Theta |\\
&\leq&|\int _Xf^*(\Omega _{i,n}^+)\wedge \Theta |+|\int _Xf^*(\Omega
_{i,n}^-)\wedge \Theta |\\
&\leq&A\int _Xf^*(\Omega _{i,n}^+)\wedge \omega _X^{dim(X)-p}+A\int _Xf^*(\Omega _{i,n}^-)\wedge \omega _X^{dim(X)-p}.
\end{eqnarray*}
Since $\Omega _{i,n}^{\pm}$ are positive closed currents with uniformly bounded norms, the last integrals are uniformly bounded as well.

ii) From i) we see that for any good approximation scheme by $C^2$ forms $K_n$, the sequence $f^*(R_{1,n})-f^*(R_{2,n})$ has a
convergent sequence. We now show that the limit is unique, hence complete the proof of Theorem \ref{TheoremInterestingExample1}. So let $\tau$ be the limit
of the sequence $f^*(R_{1,n})-f^*(R_{2,n})$. Such a $\tau$ is a $DSH^{p-1}$ current by the consideration in i). Let $H_n=H_n^+-H_n^-$
be another good approximation scheme by $C^2$ forms, and let $\tau '$ be the corresponding limit, which is in $DSH^{p-1}$. We
want to show that $\tau =\tau '$. or equivalently, to show that $\tau -\tau '=0$.

By Theorem \ref{TheoremLocalPullbackForPositiveClosedCurrents}, $\tau -\tau '=0$ in $X-\pi _X(\mathcal{C}_f)$. Hence support of $\tau -\tau '$ is
contained in $\pi _X(\mathcal{C}_f)$. Since $\tau -\tau '$ is in $DSH^{p-1}$, it is a $\mathbb{C}$-flat $(p-1,p-1)$ current (see Bassanelli
\cite{bassanelli}). Because the codimension of $\pi _X(\mathcal{C}_f)$ is $\geq p$, it follows by Federer-type support theorem for $\mathbb{C}$-flat
currents (see Theorem 1.13 in \cite{bassanelli}) that $\tau -\tau '=0$ identically.

b) Finally, we show that if $T_j$ are positive closed $(p,p)$ currents converging in $DSH$ to $T$ then $f^{\sharp}(T_j)$ weakly converges to
$f^{\sharp}(T)$.

We let $\pi _1,\pi _2:Y\times Y\rightarrow Y$ be the projections, and let $\Delta _Y$ be the diagonal. As in the proof of Lemma
\ref{LemmaSmoothRepresentativeOfPositiveClosedCurrents}, we choose $\Delta $ to be a smooth closed $(dim (Y),dim (Y))$ on $Y$ having the same cohomology
class with $[\Delta _Y]$. We write $\Delta =\Delta ^+-\Delta ^-$, where $\Delta ^{\pm}$ are strongly positive smooth closed $(dim (Y),dim (Y))$ forms. If
we define $\phi _j^{\pm}=(\pi _1)_*(\pi _2^*(T_j)\wedge \Delta ^{\pm})$ and $\phi ^{\pm}=(\pi _1)_*(\pi _2^*(T)\wedge \Delta ^{\pm})$, then
$\{T_j\}=\{\phi _j^+-\phi _j^-\}$ and $\{T\}=\{\phi ^+-\phi ^-\}$. Moreover, $\phi _j^{\pm}$ are positive closed smooth forms converging uniformly to
$\phi ^{\pm}$. Hence $f^*(\phi _j^{\pm})$ weakly converges to $f^*(\phi ^{\pm})$. Thus to show that $f^{\sharp}(T_j)$ weakly converges to
$f^{\sharp}(T)$, it is enough to show that $f^{\sharp}(T_j-\phi _j)$ weakly converges to $f^{\sharp}(T-\phi )$, where we define $\phi _j=\phi _j^+-\phi
_j^-$ and $\phi =\phi ^+-\phi ^-$.

By Proposition 2.1 in \cite{dinh-sibony5}, there are positive $(p-1,p-1)$ currents $R_j^{\pm}$ and $R^{\pm}$ so that $T_j-\phi _j=dd^c(R_j^+-R_j^-)$,
$T-\phi =dd^c(R^+-R^-)$. Moreover, we can choose these in such a way that $R_j^{\pm}$ converges in $DSH$ to $R^{\pm}$. From the proof of a), $f^{\sharp}$
is well-defined on the set of $DSH^{p-1}$ currents. Thus to prove b) we need to show only that $f^{\sharp}(R_j^{\pm})$ weakly converges to
$f^{\sharp}(R^{\pm})$.

By Theorem \ref{TheoremLocalPullbackForPositiveClosedCurrents}, on $X-\pi _X(\mathcal{C}_f)$ the currents $f^{\sharp}(R_j^{\pm})$ and
$f^{\sharp}(R^{\pm})$ are the same as the currents $f^{o}(R_j^{\pm})$ and $f^{o}(R^{\pm})$ defined in \cite{dinh-sibony2}. Hence by the results in
\cite{dinh-sibony2}, it follows that $f^{\sharp}(R_j^{\pm})$ weakly converges in $X-\pi _X(\mathcal{C}_f)$ to $f^{\sharp}(R^{\pm})$. Thus as in the proof
of a), to show that $f^{\sharp}(R_j^{\pm})$ weakly converges to $f^{\sharp}(R^{\pm})$, it suffices to show that $||f^{\sharp}(R_j)||_{DSH}$ is uniformly
bounded.

The current $f^{\sharp}(R_j)$ is the limit of $f^*(\mathcal{K}_{n}(R_j))$. As in a), we write $K_{n}(R_j)=R_{j,n}^+-R_{j,n}^-$ where
$R_{j,n}^{\pm}$ are positive $DSH^{p-1}(Y)$ forms of class $C^2$. Moreover, by Theorem \ref{TheoremApproximationOfDinhAndSibony}, there is a
constant $A>0$ independent of $j$ and $n$ so that $||R_{j,n}^{\pm}||_{DSH}\leq A||R_j^{\pm}||_{DSH}$. It can be seen from the
proof of a) that $f^{\sharp}(R_j)$ is a $DSH^{p-1}$ current. Moreover $||f^{\sharp}(dd^cR_j)||_{DSH}$, which can be bound using intersections of
cohomology classes, is $\leq A||R_j||_{DSH}$, where $A>0$ is independent of $j$.

We choose an open neighborhood $V$ of $\pi _X(\mathcal{C}_f)$ and a form $\Phi$ as in the proof of a). Then we can see from a) that
\begin{eqnarray*}
||f^{\sharp}(R_j)||_{DSH}\leq A||f^{\sharp}(R_j)||_{X-V,DSH}+A||f^{\sharp}(dd^cR_j)||_{DSH},
\end{eqnarray*}
where $A>0$ is a constant independent of $j$, and $||f^{\sharp}(R_j)||_{X-V,DSH}$ means the $DSH$ norm of $f^{\sharp}(R_j)$ computed on the set $X-V$.
From the results in \cite{dinh-sibony2}, $||f^{\sharp}(R_j)||_{X-V,DSH}$ is uniformly bounded. The term $||f^{\sharp}(dd^cR_j)||_{DSH}$ was shown above
to be uniformly bounded as well. Thus $||f^{\sharp}(R_j)||_{DSH}$ is uniformly bounded as desired.
\end{proof}
\begin{proof} (of Theorem \ref{TheoremInterestingExample})

 Let $\theta$ be a closed smooth form on $Y$ having the same cohomology class as $T$. Since $T$ is continuous on $U=X-A$, there are $DSH^{p-1}$ currents $R^{\pm}$ so that $T-\theta =dd^c(R^+)-dd^c(R^-)$, where $R^{\pm}|_U$ are continuous (see Proposition 2.1 in \cite{dinh-sibony5}).  As in the proof of the Theorem \ref{TheoremInterestingExample1}, we will show that $f^{\sharp}(R^{\pm})$ are well-defined. Since $f^{-1}(A)\cap \pi _X(\mathcal{C}_f)\subset V$, where $V$ is of codimension $\geq p$, it is enough as before to show that $f^{*}(K_{n}^{\pm}(R^{\pm}))$ have bounded masses outside a small neighborhood of $f^{-1}(A)\cap \pi _X(\mathcal{C}_f)$. First, by the proof of Theorem \ref{TheoremInterestingExample1},  $f^{*}(K_{n}^{\pm}(R^{\pm}))$ have bounded masses outside a small neighborhood of $\pi _X(\mathcal{C}_f)$. Hence it remains to show that  $f^{*}(K_{n}^{\pm}(R^{\pm}))$ have bounded masses outside a small neighborhood of $f^{-1}(A)$.

Let $B$ be a small neighborhood of $f^{-1}(A)$. Then there is a cutoff function $\chi $ for $A$, so that $f^{-1}(supp(\chi ))\subset B$. We write
\begin{eqnarray*}
f^{*}(K_{n}^{\pm}(R^{\pm}))=f^{*}(\chi K_{n}^{\pm}(R^{\pm}))+f^{*}((1-\chi )K_{n}^{\pm}(R^{\pm})).
\end{eqnarray*}
The first current has support in $B$, and hence has no contribution for the mass of $f^{*}(K_{n}^{\pm}(R^{\pm}))$ outside $B$. By properties of good approximation schemes by $C^2$ forms, $(1-\chi )K_{n}^{\pm}(R^{\pm})$ uniformly converges to a continuous form on $Y$, and hence $f^{*}((1-\chi )K_{n}^{\pm}(R^{\pm}))$ has uniformly bounded masses on $X$, which is what wanted to prove.

To complete the proof, we need to show the continuity stated in the theorem. This continuity can be proved using the arguments from the first part of the
proof, and from part b) of the proof of Theorem \ref{TheoremInterestingExample1} and the proof of Proposition \ref{PropositionUniformlyApproximation}.
\end{proof}

\begin{proof} (Of Theorem \ref{TheoremUseSiuDecomposition})

By assumption and Corollary \ref{CorollaryPullBackVariety}, if $V$ is an analytic variety of codimension $p$ contained in $E(T)$, then $f^{\sharp}[V]$ is
well-defined with the number $s=0$ in Definition \ref{DefinitionPullbackDdcOfOrderSCurrents}. Hence the currents
\begin{eqnarray*}
W_N=\sum _{j=1}^n\lambda _j[V_j]
\end{eqnarray*}
can be pulled back with the same number $s=0$ in Definition \ref{DefinitionPullbackDdcOfOrderSCurrents}, here $N$ is a positive integer. Since $0\leq
\sum _j\lambda _j[V_j] -W_N=S_N$ where $S_N\rightharpoonup 0$ as $N\rightarrow \infty$, by Theorem \ref{TheoremContinuityPropertiesOfPullbackOperator} it
follows that $f^{\sharp}(\sum _j\lambda _j[V_j])=\sum _j\lambda _jf^{\sharp}[V_j]$ is well-defined.
\end{proof}

\begin{proof} (Of Theorem \ref{TheoremPullbackOfMeasures})
Let $T$ be a positive measure on $Y$ having no mass on $\pi _Y(\mathcal{C}_f)$. Let $K_n$ be a good approximation scheme by $C^2$ forms. Then we will show that as $n$ converges to $\infty$,
any limit point of $[\Gamma _f]\wedge \pi _Y^*(K_{n}(T))$ has no mass on $\mathcal{C}_f$. Thus $\lim _{n\rightarrow\infty}[\Gamma _f]\wedge \pi
_Y^*(K_{n}(T))=(\pi _Y|_{\Gamma _f})^*(T)$ where the RHS is defined in \cite{dinh-sibony2}. Then $f^{\sharp}(T)$ is well-defined, and moreover
equals to the current $f^o(T)$ defined in \cite{dinh-sibony2}, thus satisfies all the conclusions of Theorem \ref{TheoremPullbackOfMeasures}.

Now we proceed to prove that any limit point $\tau $ of $[\Gamma _f]\wedge \pi _Y^*(K_{n}(T))$ has no mass on $\mathcal{C}_f$. This is
equivalent to showing that for a smooth $(dim (X),dim(X))$ form $\alpha$ on $X\times Y$, and for a sequence $\theta _j$ of smooth functions on $X\times
Y$ having the properties: $0\leq \theta _j\leq 1$, $\theta _j=1$ on a neighborhood of $\mathcal{C}_f$, and support of $\theta _j$ converges to
$\mathcal{C}_f$ then:
\begin{eqnarray*}
\lim _{j\rightarrow\infty}\lim _{n\rightarrow\infty}\int _{X\times Y}\theta _j\alpha \wedge [\Gamma _f]\wedge \pi _Y^*(K_{n}(T))=0.
\end{eqnarray*}
By properties of good approximation schemes by $C^2$ forms, we can write the above equality as
\begin{equation}
\lim _{j\rightarrow\infty}\lim _{n\rightarrow\infty}\int _{X\times Y}T\wedge K_{n}((\pi _Y)_*(\theta _j\alpha \wedge [\Gamma
_f]))=0.\label{EquationTheoremPullbackMeasures1}
\end{equation}
Writing $\alpha$ as the difference of two positive smooth forms, we may assume that
$\alpha$ is positive. Now $\alpha$ is a positive smooth form, since $0\leq \theta _j\leq 1$ for all $j$, we can bound the function $(\pi _Y)_*(\theta
_j\alpha \wedge [\Gamma _f])$ by a multiplicity of $(\pi _Y)_*(\omega _{X\times Y}^{dim(X)} \wedge [\Gamma _f])$ independently of $j$. The later is a
constant, thus $(\pi _Y)_*(\theta _j\alpha \wedge [\Gamma _f])$ is a positive bounded function. Then $K_{n}((\pi
_Y)_*(\theta _j\alpha \wedge [\Gamma _f]))$ are $C^2$ functions uniformly bounded w.r.t. $j$ and $n$. Moreover,
the support of $K_{n}((\pi _Y)_*(\theta _j\alpha \wedge [\Gamma _f]))$ converges to $\pi _Y(\mathcal{C}_f)$ as $j\rightarrow\infty$,
independent of $n$. Because $T$ has no mass on $\pi _Y(\mathcal{C}_f)$, we can then apply Lebesgue's dominated convergence theorem
to obtain (\ref{EquationTheoremPullbackMeasures1}).
\end{proof}

\section{The map $J_X$}
Through out this section, let $X$ be the blowup of $\mathbb{P}^3$ along $4$ points $e_0=[1:0:0:0],e_1=[0:1:0:0],e_2=[0:0:1:0], e_3=[0:0:0:1]$;
$J:\mathbb{P}^3\rightarrow \mathbb{P}^3$ is the Cremona map $J[x_0:x_1:x_2:x_3]=[1/x_0:1/x_1:1/x_2:1/x_3]$, and let $J_X$ be the lifting of $J$ to $X$.
For $0\leq i\not= j\leq 3$, $\Sigma _{i,j}$ is the line in $\mathbb{P}^3$ consisting of points $[x_0:x_1:x_2:x_3]$ where $x_i=x_j=0$, and
$\widetilde{\Sigma _{i,j}}$ is the strict transform of $\Sigma _{i,j}$ in $X$.

Let $E_0,E_1,E_2,E_3$ be the corresponding exceptional divisors of the blowup $X\rightarrow \mathbb{P}^3$, and let $L_0,L_1,L_2,L_3$ be any lines in
$E_0,E_1,E_2,E_3$ correspondingly. Let $H$ be a generic hyperplane in $\mathbb{P}^3$, and let $H^2$ be a generic line in $\mathbb{P}^3$. Then
$H,E_0,E_1,E_2,E_3$ are a basis for $H^{1,1}(X)$, and $H^2,L_0,L_1,L_2,L_3$ are a basis for $H^{2,2}(X)$. Intersection products in complementary
dimensions are (see for example Chapter 4 in \cite{griffiths-harris}):
\begin{eqnarray*}
&&H.H^2=1,~H.L_0=0,~H.L_1=0,~H.L_2=0,~H.L_3=0,\\
&&E_0.H^2=0,~E_0.L_0=-1,~E_0.L_1=0,~E_0.L_2=0,~E_0.L_3=0,\\
&&E_1.H^2=0,~E_1.L_0=0,~E_1.L_1=-1,~E_1.L_2=0,~E_1.L_3=0,\\
&&E_2.H^2=0,~E_2.L_0=0,~E_2.L_1=0,~E_2.L_2=-1,~E_1.L_3=0,\\
&&E_3.H^2=0,~E_3.L_0=0,~E_3.L_1=0,~E_3.L_2=0,~E_3.L_3=-1.
\end{eqnarray*}
The map $J_X^*:H^{1,1}(X)\rightarrow H^{1,1}(X)$ is not hard to compute (see for example the computations in Example 2.5 in \cite{guedj}):
\begin{eqnarray*}
J_X^*(H)&=&3H-2E_0-2E_1-2E_2-2E_3,\\
J_X^*(E_0)&=&H-E_1-E_2-E_3,\\
J_X^*(E_1)&=&H-E_0-E_2-E_3,\\
J_X^*(E_2)&=&H-E_0-E_1-E_3,\\
J_X^*(E_3)&=&H-E_0-E_1-E_2.
\end{eqnarray*}
If $x\in H^{1,1}(X)$ and $y\in H^{2,2}(X)$, since $J_X^2=$the identity map on $X$, we have the duality $(J_X^*y).x=y.(J_X^*x)$. Thus from the above data,
we can write down the map $J_X^*:H^{2,2}(X)\rightarrow H^{2,2}(X)$:
\begin{eqnarray*}
J_X^*(H^2)&=&3H^2-L_0-L_1-L_2-L_3,\\
J_X^*(L_0)&=&2H^2-L_1-L_2-L_3,\\
J_X^*(L_1)&=&2H^2-L_0-L_2-L_3,\\
J_X^*(L_2)&=&2H^2-L_0-L_1-L_3,\\
J_X^*(L_3)&=&2H^2-L_0-L_1-L_2.
\end{eqnarray*}
Now we are ready to prove Corollary \ref{CorollaryTheMapJ}.
\begin{proof} (of Corollary \ref{CorollaryTheMapJ})
The restriction $J_X:X-\bigcup \widetilde{\Sigma _{i,j}}\rightarrow X-\bigcup \widetilde{\Sigma _{i,j}}$ is a biholomorphic map, as can be seen by using
local coordinate projections for the blowup $\pi$ near the exceptional divisors $E_i$'s. Moreover it can be shown that $J_X(\widetilde{\Sigma
_{i,j}})=\widetilde{\Sigma _{3-i,3-j}}$, and every point on $\widetilde{\Sigma _{i,j}}$ blows up to $\widetilde{\Sigma _{3-i,3-j}}$. Hence $\pi
_1(\mathcal{C}_{J_X})= \bigcup \widetilde{\Sigma _{i,j}}$. Therefore the map $J_X$ satisfies Theorem \ref{TheoremInterestingExample1} for $p=2$. Thus if
$T$ is a positive closed $(2,2)$ current on $X$ then $J_X^{\sharp}(T)$ is well-defined. For an alternative proof of this fact, see Lemma
\ref{LemmaOtherGoodPropertyOfTheMapJ} below.

It remains to show that $J_X^{\sharp}[\widetilde{\Sigma _{0,1}}]=-[\widetilde{\Sigma _{2,3}}]$. Since $J_X^{-1}(\widetilde{\Sigma
_{0,1}})=\widetilde{\Sigma _{2,3}}$, by Theorem \ref{TheoremInterestingExample} there is a number $\lambda $ so that $J_X^{\sharp}[\widetilde{\Sigma
_{0,1}}]=\lambda [\widetilde{\Sigma _{2,3}}]$. To determine $\lambda$, we need to know $J_X^{*}\{\widetilde{\Sigma _{0,1}}\}$. We have
$\{\widetilde{\Sigma _{0,1}}\}=\{H^2-L_2-L_3\}$, hence from the above data we have
$$J_X^{*}\{\widetilde{\Sigma _{0,1}}\}=J_X^{*}\{H^2\}-J_X^{*}\{L_2\}-J_X^*\{L_3\}=\{-H^2+L_0+L_1\}=-\{\widetilde{\Sigma _{2,3}}\},$$
thus $\lambda =-1$, and $J_X^{\sharp}[\widetilde{\Sigma _{0,1}}]=- [\widetilde{\Sigma _{2,3}}]$.
\end{proof}

The following result gives an alternative proof to the conclusions of Corollary \ref{CorollaryTheMapJ}. In its proof we will make use of the space $Y$
defined in the statement of Proposition \ref{PropositionTheOperatorDotPullback} below. Here $\pi :Y\rightarrow X$ is the blowup of $X$ along all
submanifolds $\widetilde{\Sigma _{i,j}}$ ($1\leq i<j\leq 3$). Then the lifting map $J_Y$ of $J$ to $Y$ is an involutive automorphism. Moreover, if we let
$S_{i,j}$ denote the exceptional divisor of $Y$ over $\widetilde{\Sigma _{i,j}}$, then $J_Y(S_{0,1})=S_{2,3}$, $J_Y(S_{0,2})=S_{1,3}$, and
$J_Y(S_{0,3})=S_{1,2}$.

\begin{lemma}
Let $T_n^+$ and $T_n^-$ be positive closed smooth $(2,2)$ forms on $X$, so that

i) $||T_n^+||,~||T_n^-||$ are uniformly bounded,

and

ii) $T_n^+-T_n^-\rightharpoonup [\widetilde{\Sigma _{0,1}}]$.

Then $J_X^*(T_n^+-T_n^-)\rightharpoonup -[\widetilde{\Sigma _{2,3}}]$.

As a consequence, if we replace $[\widetilde{\Sigma _{0,1}}]$ in i) and ii) above by any positive closed $(2,2)$ current $T$ then $J_X^*(T_n^+-T_n^-)$
converges to $J_X^{\sharp}(T)$. \label{LemmaOtherGoodPropertyOfTheMapJ}\end{lemma}

\begin{proof}
Let $\tau _n^+=\pi ^*(T_n^+)$ and $\tau _n^-=\pi ^*(T_n^-)$, which are positive closed currents on $Y$. By assumption i), $||\tau _n^+||$ and $||\tau
_n^-||$ are uniformly bounded. Thus we may assume that $\tau _n^+\rightharpoonup \tau ^+$ and $\tau _n^-\rightharpoonup \tau ^-$, where $\tau ^+$ and
$\tau ^-$ are positive closed currents on $Y$.

Since $J_Y$ is a biholomorphic map, we can pull-back any current on $Y$ by $J_Y$. It is not hard to see that
\begin{eqnarray*}
J_X^*(T_n^+)&=&\pi _*(J_Y^* \tau _n^+),\\
J_X^*(T_n^-)&=&\pi _*(J_Y^* \tau _n^-).
\end{eqnarray*}
Hence
\begin{eqnarray*}
J_X^*(T_n^+-T_n^-)\rightharpoonup \pi _*(J_Y^*(\tau ^+-\tau ^-)).
\end{eqnarray*}

We need to show that the latter current is $-[\widetilde{\Sigma _{2,3}}]$. To this end, it suffices to
show that support of $\pi _*(J_Y^*(\tau ^+-\tau ^-))$ is in $\widetilde{\Sigma _{2,3}}$. In fact, then we will have $\pi _*(J_Y^*(\tau ^+-\tau ^-))=\lambda [\widetilde{\Sigma _{2,3}}]$, and the computation on cohomology shows that $\lambda =-1$.

It is not hard to see that support of $\tau ^+-\tau ^-$ is contained in the union of $S_{i,j}$'s ($1\leq i<j\leq 3$). Let $\tau _{i,j}=\tau
^+|_{S_{i,j}}-\tau ^-|_{S_{i,j}}$ with support in $S_{i,j}$ so that $\tau =\sum _{1\leq i<j\leq 3}\tau _{i,j}$. In $H^{2,2}(Y)$ we have:
\begin{eqnarray*}
\pi ^*\{\widetilde{\Sigma _{0,1}}\}=\{\tau ^+-\tau ^-\}=\sum _{i,j}\{\tau _{i,j}\},
\end{eqnarray*}
here $\pi ^*\{\widetilde{\Sigma _{0,1}}\}$ can be represented by currents with support in $S_{0,1}$. Moreover, by considering the push-forwards $\pi
_*(\tau _n^{+}-\tau _n^{-})$, it follows that $\pi _*(\tau _{i,j})=0$ where $(i,j)\not= (0,1)$. It can be checked that each fiber $S_{i,j}$ is a product
$S_{i,j}\simeq \mathbb{P}^1\times \mathbb{P}^1$, hence by Kuneth's theorem $H^{2,2}(S_{i,j})$ is generated by a "horizontal curve" $\alpha _{i,j}$  and a
"vertical curve" (or fiber) $\beta _{i,j}$. Here the properties of "horizontal curve" and "vertical curve" that we use are that $\pi _{*}(\alpha
_{i,j})=\widetilde{\Sigma}_{i,j}$ and $\pi _{*}(\beta _{i,j})=0$. Hence there are numbers $a_{i,j}$ and $b_{i,j}$ so that the cohomology class of $\tau
_{i,j}-a_{i,j}\alpha _{i,j}-b_{i,j}\beta _{i,j}$ is zero. For $(i,j)\not= (0,1)$, since $\pi _{*}(\tau _{i,j})=0$, it follows that
\begin{eqnarray*}
a_{i,j}\{\widetilde{\Sigma}_{i,j}\}=\pi _*\{a_{i,j}\alpha _{i,j}+b_{i,j}\beta _{i,j}\}=\pi _*\{\tau _{i,j}\}=\{\pi _*({\tau _{i,j}})\}=0.
\end{eqnarray*}
Hence $a_{i,j}=0$ for $(i,j)\not= (0,1)$.

Note that a non-zero $(2,2)$-cohomology class in $H^{2,2}(Y)$ represented by currents with supports in $S_{0,1}$ can not be represented by a linear
combinations of "vertical curves" with support in $\bigcup _{(i,j)\not= (0,1)}S_{i,j}$: Assume that
\begin{eqnarray*}
\{a_{0,1}\alpha _{0,1}+b_{0,1}\beta _{0,1}+\sum _{(i,j)\not= (0,1)}b_{i,j}\beta _{i,j}\}=0
\end{eqnarray*}
in $H^{2,2}(Y)$. Push-forward by the map $\pi$ implies that $a_{0,1}\{\widetilde{\Sigma}_{0,1}\}=0$ in $H^{2,2}(X)$, and hence $a_{0,1}=0$. Thus $\{\sum
b_{i,j}\beta _{i,j}\}=0$ in $H^{2,2}(Y)$. Use the fact that $\{S_{i,j}\}.\{\beta _{k,l}\}=-1$ if $(k,l)=(i,j)$, and $=0$ otherwise (see for example
Chapter 4 in \cite{griffiths-harris}), we imply that $b_{i,j}=0$ for all $(i,j)$ as claimed.

Hence it follows that $\{\tau _{i,j}\}=0$ in $H^{2,2}(Y)$ for $(i,j)\not= (0,1)$.

We have
\begin{eqnarray*}
\pi _*(J_Y^*(\tau ^+-\tau ^-))=\sum _{i,j}\pi _*(J_Y^*\tau _{i,j}),
\end{eqnarray*}
where support of $\pi _*(J_Y^*\tau _{i,j})$ is contained in $\widetilde{\Sigma _{3-i,3-j}}$. Here we use the convention that $\widetilde{\Sigma
_{k,l}}:=\widetilde{\Sigma _{l,k}}$ if $k>l$. Since $\pi _*(J_Y^*\tau _{i,j})$ is a normal $(2,2)$ current, it follows from the structure theorem for
normal currents that there is $\lambda _{i,j}\in \mathbb{R}$ so that $\pi _*(J_Y^*\tau _{i,j})=\lambda _{i,j}[\widetilde{\Sigma _{3-i,3-j}}]$. If
$(i,j)\not= (0,1)$ then $\{\tau _{i,j}\}=0$ in $H^{2,2}(Y)$, thus $\{\pi _*(J_Y^*\tau _{i,j})\}=0$ in $H^{2,2}(X)$, which implies $\lambda _{i,j}=0$ for
such $(i,j)$'s. Hence
\begin{eqnarray*}
\pi _*(J_Y^*(\tau ^+-\tau ^-))=\pi _*(J_Y^*\tau _{0,1})
\end{eqnarray*}
has support in $\widetilde{\Sigma _{2,3}}$ as wanted.
\end{proof}
\begin{proposition}
Let $X$ be the space constructed in Corollary \ref{CorollaryTheMapJ}. Let $\pi :Y\rightarrow X$ be the blowup of $X$ along all submanifolds
$\widetilde{\Sigma _{i,j}}$ ($1\leq i<j\leq 3$). Then there is a positive closed $(2,2)$-current $T$ on $X$ with $L^1$ coefficients so that: in
$H^{2,2}(Y),$

$$\{\pi ^o(T)\}\not= \pi ^*\{T\}.$$
Here the operator $\pi ^o$ is defined in Dinh and Nguyen \cite{dinh-nguyen}. In this case, in fact $\pi ^o(T)$ is also the operator defined in Dinh and
Sibony \cite{dinh-sibony2}.

\label{PropositionTheOperatorDotPullback}\end{proposition}
\begin{proof}

We assume in order to reach a contradiction that for any positive closed $(2,2)$ current $T$ on $X$ with $L^1$-coefficients then $\{\pi
^o(T)\}=\pi^*\{T\}$ in $H^{2,2}(Y)$.

By regularization theorem of Dinh and Sibony, there is a sequence $T_n^+$ and $T^-$ of positive closed $(2,2)$ currents with $L^1$-coefficients such that
$||T_n^+||$ are uniformly bounded and $T_n^+\rightharpoonup T^-+ [\widetilde{\Sigma _{0,1}}]$. By the assumption we have $\{\pi
^o(T_n^+)\}=\pi^*\{T_n^+\}$ for any $n$, and $\{\pi ^o(T^-)\}=\pi^*\{T^-\}$. Now for the maps $J_X$ and $J_Y$ considered above, it is not hard to see
that $J_X^o=\pi _*J_Y^*\pi ^o$. Thus, we also have $\{J_X ^o(T_n^+)\}=J_X^*\{T_n^+\}$ and $\{J_X ^o(T^-)\}=J_X ^*\{T^-\}$.

Let $\tau ^+$ be a cluster point of $J_X ^o(T_n^+)$. Then it is easy to see that

$$\tau ^+\geq J_X^o(T^-+[\widetilde{\Sigma _{0,1}}])=J_X ^o(T^-)+J_X ^o([\widetilde{\Sigma _{0,1}}])=J_X ^o(T^-).$$
But then this contradicts the fact that in $H^{2,2}(X)$:
\begin{eqnarray*}
\{\tau ^+\}&=&\lim _{n\rightarrow \infty}\{J_X^o(T_n^+)\}=\lim _{n\rightarrow \infty}J_X^*\{(T_n^+)\}\\
&=&J_X^*\{T^-\}+J_X^*\{\widetilde{\Sigma _{0,1}}\}=\{J_X^o(T^-)\}-\{[\widetilde{\Sigma _{2,3}}]\},
\end{eqnarray*}
here we used the assumption that $J_X^*\{(T_n^+)\}=\{J_X^o(T_n^+)\}$ and $J_X^*\{(T^-)\}=\{J_X^o(T^-)\}$.
\end{proof}

\begin{proposition}
Let $X$ be the space constructed in Corollary \ref{CorollaryTheMapJ}. There is no sequence $T_n^+$ and $T^-$ of positive closed smooth $(2,2)$ forms on
$X$ such that

i) $||T_n^+||$ are uniformly bounded

ii) $T_n^+-T^- \rightharpoonup [\widetilde{\Sigma _{0,1}}]$. \label{PropositionNoVeryGoodApproximation}\end{proposition}

\begin{remark}
In Example 6.3 of the paper \cite{bost-gillet-soule} of Bost, Gillet, and Soule, a related result was given.
\end{remark}
\begin{proof}
Use the same argument as that in the proof of Proposition \ref{PropositionTheOperatorDotPullback}, but now use that if $T_n^{\pm}$ are positive closed
smooth forms then $J_X^{*}(T_n^{\pm})=J_X^o(T_n^{\pm})$, and hence $\{J_X^o(T_n^{\pm})\}=J_X^*\{T_n^{\pm}\}$.
\end{proof}
\section{Invariant currents}
Throughout this section, we let $X$ be a compact K\"ahler manifold of dimension $k$, and let $f:X\rightarrow X$ be a dominant meromorphic map.

We introduce in the below a condition, called $dd^c$-$p$ stability. This condition seems to be natural for the problem of finding invariant $(p,p)$
currents for a self-map $f$ (see the discussions and the results after the definition).
\begin{definition}
We say that $f$ satisfies the $dd^c$-$p$ stability condition if the following holds: For any smooth $(p-1,p-1)$ form $\alpha$ and for any $n$,
$f^{\sharp}((f^n)^*dd^c \alpha )$ is well-defined, and moreover $f^{\sharp}((f^n)^*dd^c\alpha )=(f^{n+1})^*(dd^c\alpha )$.
 \label{DefinitionDdcPStability}\end{definition}

In general, condition of $dd^c$-$p$ stability has no relation with condition of $p$-algebraic stability. On the one hand, the $dd^c$-$p$ stability
condition requires no constraints on the action of $f^*$ on $H^{p,p}(X)$, because the cohomology class of $dd^c(\alpha )$ is zero. On the other hand, it
asks for the possibility of iterated pull-back $dd^c(\alpha )$ by $f$. Any map $f$ is $dd^c$-$1$ stable, whether being or not $1$-algebraic stable. If
$f$ is $p$-analytic stable then $f$ is $dd^c$-$p$ stable. Using the method in Step 1 of the proof of Lemma \ref{LemmaJXSquare}, it can be shown that the
linear pseudo-automorphisms in \cite{bedford-kim} are $dd^c$-$2$ stable. We suspect that these pseudo-automorphisms are also $2$-analytic stable even
though it seems not be easily checked.

We first introduce an abstract result on invariant $(p,p)$ currents.

\begin{theorem}
Assume that $f:X\rightarrow X$ satisfies the $dd^c$-$p$ condition. Let $\lambda$ be a real eigenvalue of $f^*:H^{p,p}(X)\rightarrow H^{p,p}(X)$, and let
$0\not= \theta _{\lambda}\in H^{p,p}(X)$ be an eigenvector with eigenvalue $\lambda$. Assume moreover that $|\lambda |>\delta _{p-1}(f)$ and let $s\geq
2$ be an integer. Then any of the following statements is equivalent to each other:

1) There is a closed $(p,p)$ current $T$ of order $s$ with $\{T\}=\theta _{\lambda }$ so that $f^{\sharp}(T)$ is well-defined, and moreover
$f^{\sharp}(T)=\lambda T$.

2) There are a smooth $(p-1,p-1)$ form $\alpha$ and a closed $(p,p)$ current $T$ of order $s$ with $\{T\}=\theta _{\lambda }$ so that $f^{\sharp}(T)$ is
well-defined, and moreover $f^{\sharp}(T)=\lambda T +\lambda dd^c(\alpha )$.

3) For any smooth $(p-1,p-1)$ form $\alpha$, there is a closed $(p,p)$ current $T$ of order $s$ with $\{T\}=\theta _{\lambda }$ so that $f^{\sharp}(T)$
is well-defined, and moreover $f^{\sharp}(T)=\lambda T +\lambda dd^c(\alpha )$.

4) There is a closed $(p,p)$ current $T$ of order $s$ with $\{T\}=\theta _{\lambda }$ so that $f^{\sharp}(T)$ is well-defined, and moreover
$f^{\sharp}(T)-\lambda T $ is a smooth form.

\label{TheoremInvariantCurrents}\end{theorem}

Note that for the current $T$ in Theorem \ref{TheoremInvariantCurrents}, we do not know whether $(f^n)^{\sharp}(T)$ (for $n\geq 2$) is well-defined or
not. The proof of Theorem \ref{TheoremInvariantCurrents} makes use of the following result, which is interesting in itself.
\begin{theorem}
Let $T_j$ and $T$ be $(p,p)$ currents of order $s_0$. Assume that $-S_j\leq T-T_j\leq S_j$ for any $j$, where $S_j$ are positive closed $(p,p)$ currents
with $||S_j||\rightarrow 0$ as $j\rightarrow \infty$.

1) If $f^{\sharp}(T_j)$ is well-defined for any $j$ with the same number $s$ in Definition \ref{DefinitionPullbackDdcOfOrderSCurrents}, then
$f^{\sharp}(T)$ is well-defined. Moreover $f^{\sharp}(T_j)$ weakly converges to $f^{\sharp}(T)$.

2) If $f^{\sharp}(dd^cT_j)$ is well-defined for any $j$ with the same number $s$ in Definition \ref{DefinitionPullbackDdcOfOrderSCurrents}, then
$f^{\sharp}(dd^cT)$ is well-defined. Moreover $f^{\sharp}(dd^cT_j)$ weakly converges to $f^{\sharp}(dd^cT)$.
 \label{TheoremContinuityPropertiesOfPullbackOperator}\end{theorem}
Note that when $p=0$, a closed $(0,0)$ current on $X$ is a constant, hence the $S_j$ in Theorem \ref{TheoremContinuityPropertiesOfPullbackOperator} are
positive constants converging to zero.

\begin{proof} (Of Theorem \ref{TheoremContinuityPropertiesOfPullbackOperator})

i) Let $K_n=K^{+}_{n}-K^-_n$ be a good approximation scheme by $C^{s+2}$ forms. Let $\alpha$ be a strongly positive smooth $(k-p,k-p)$ form on $X$. then
$f_*(\alpha)$ is a strongly positive form. Therefore $K_{n}^{\pm}f_*(\alpha )$ are strongly positive forms of class $C^2$. Since $-S_j\leq T_j-T\leq
S_j$, by Lemma \ref{LemmaIntersectionOfPositiveContinuousForms} we obtain
\begin{eqnarray*}
-\int _{X}S_j\wedge K_{n}^{\pm}f_*(\alpha )\leq \int _{X}(T_j-T)\wedge K_{n}^{\pm}f_*(\alpha )\leq \int _{X}S_j\wedge
K_{n}^{\pm}f_*(\alpha ).
\end{eqnarray*}
From Lemma \ref{LemmaBoundContinuousForms}, there is a constant $A>0$ independent of $\alpha$ so that $A||\alpha ||_{L^{\infty}}\omega _X^{k-p}\pm
\alpha$ are strongly positive forms. Then $A||\alpha ||_{L^{\infty}}f_*(\omega _X^{k-p})\pm f_*(\alpha )$ are strongly positive forms on $X$. Hence we
have
\begin{eqnarray*}
\int _{X}S_j\wedge K_{n}^{\pm}f_*(\alpha )\leq A||\alpha ||_{L^{\infty}}\int _{X}S_j\wedge K_{n}^{\pm}f_*(\omega _X^{k-p}).
\end{eqnarray*}
The latter integral can be computed cohomologously, hence can be bound as
\begin{eqnarray*}
A||\alpha ||_{L^{\infty}}\int _{X}S_j\wedge K_{n}^{\pm}f_*(\omega _X^{k-p})&\leq& A||\alpha ||_{L^{\infty}}||S_j||\times
||K_{n}^{\pm}f_*(\omega _X^{k-p})||\\
&\leq&A||\alpha ||_{L^{\infty}}||S_j||\times ||f_*(\omega _X^{k-p})||.
\end{eqnarray*}
The latter inequality comes from Theorem \ref{TheoremApproximationOfDinhAndSibony}.  Hence,
\begin{equation}
-A||\alpha ||_{L^{\infty}}||S_j||\leq \int _{X}(T_j-T)\wedge K_{n}^{\pm}f_*(\alpha )\leq A||\alpha ||_{L^{\infty}}||S_j||.
\label{EquationTheoremContinuityProperties.1}
\end{equation}
Since $f^{\sharp}(T_j)$ are well-defined for all $j$, if we take limit as $n\rightarrow \infty$ in (\ref{EquationTheoremContinuityProperties.1}), we get
\begin{eqnarray*}
-A||\alpha ||_{L^{\infty}}||S_j||&\leq& \int _{X}f^{\sharp}(T_j)\wedge \alpha-\limsup_{n\rightarrow\infty}\int _{X}T\wedge
K_{n}f_*(\alpha )\\
&\leq&\int _{X}f^{\sharp}(T_j)\wedge \alpha-\liminf_{n\rightarrow\infty}\int _{X}T\wedge
K_{n}f_*(\alpha )\\
&\leq& A||\alpha ||_{L^{\infty}}||S_j||.
\end{eqnarray*}
Since $||S_j||\rightarrow 0$, taking limit as $j\rightarrow \infty$ shows that
\begin{eqnarray*}
L(\alpha ):=\lim_{n\rightarrow\infty}\int _{X}T\wedge K_{n}f_*(\alpha )
\end{eqnarray*}
exists, and moreover it satisfies
\begin{equation}
-A||\alpha ||_{L^{\infty}}||S_j||\leq \int _Xf^{\sharp}(T_j)\wedge \alpha -L(\alpha )\leq A||\alpha
||_{L^{\infty}}||S_j||,\label{EquationTheoremContinuityProperties.2}
\end{equation}
for all $j$, and all strongly positive smooth $(dim(X)-p,dim(X)-p)$ form $\alpha$. Since any smooth $(dim (X)-p,dim (X)-p)$ form $\alpha$ is the
difference of two strongly positive smooth $(dim(X)-p,dim(X)-p)$ forms $\alpha _1$ and $\alpha _2$ whose $L^{\infty}$ norms are uniformly bounded (up to
a multiplicative constant) by $||\alpha ||_{L^{\infty}}$ by Lemma \ref{LemmaBoundContinuousForms}, it follows that
(\ref{EquationTheoremContinuityProperties.2}) holds for any smooth form $\alpha$. From this, it follows easily that the assignment $\alpha \mapsto
L(\alpha )$ is a well-defined functional on smooth forms $\alpha$. Now we show that it is a current on $X$. For this end, it suffices to show that if
$\alpha _n$ are smooth forms so that $||\alpha _n||_{C^s}\rightarrow 0$ for any fixed $s\geq 0$ then $L(\alpha _n)\rightarrow 0$. This follows easily
from (\ref{EquationTheoremContinuityProperties.2}) by first taking limit when $n\rightarrow \infty$ and then taking limit when $j\rightarrow \infty$,
using the assumptions that $f^{\sharp}(T_j)$ are currents, hence
\begin{eqnarray*}
\lim _{n\rightarrow \infty}\int _Xf^{\sharp}(T_j)\wedge \alpha _n=0,
\end{eqnarray*}
for any $j$.

ii) The proof is similar to the proof of i), with a small change: The estimate (\ref{EquationTheoremContinuityProperties.1}) is modified to
\begin{eqnarray*}
-A||dd^c\alpha ||_{L^{\infty}}||S_j||\leq \int _{X}(T_j-T)\wedge K_{n}^{\pm}f_*(dd^c\alpha )\leq A||dd^c\alpha ||_{L^{\infty}}||S_j||.
\end{eqnarray*}
\end{proof}
The proof of Theorem \ref{TheoremInvariantCurrents} also uses the following result:
\begin{lemma}
Assume that $f$ satisfies the $dd^c$-$p$ stability condition. Let $\lambda$ be a positive real number. If $|\lambda |>\delta _{p-1}(f)$, then for any
smooth $(p-1,p-1)$ form $\alpha$, there is a current $R_{\alpha}$ of order $0$, so that $f^{\sharp}(dd^cR_{\alpha })$ is well-defined, and moreover
\begin{eqnarray*}
f^{\sharp}(dd^cR_{\alpha})-\lambda dd^cR_{\alpha }=\lambda dd^c\alpha .
\end{eqnarray*}
\label{LemmaSolvingTheDdcForPullback}\end{lemma}
\begin{proof}
Define $\beta =-\alpha$, and consider
\begin{eqnarray*}
R_n=\sum _{j=0}^n\frac{(f^j)^*(\beta )}{\lambda ^j}.
\end{eqnarray*}
Since $\beta $ is a smooth $(p-1,p-1)$ form, there is a constant $A>0$ so that $-A\omega _X^{p-1}\leq \beta \leq A\omega _X^{p-1}$. It follows that
\begin{eqnarray*}
R_{\alpha}=\sum _{j=0}^{\infty}\frac{(f^j)^*(\beta )}{\lambda ^j}
\end{eqnarray*}
is a well-defined current which is a difference of two positive currents, hence of order $0$. Moreover $-S_n\leq R_n-R\leq S_n$, where
\begin{eqnarray*}
S_n=A\sum _{j=n+1}^{\infty}\frac{(f^j)^*(\omega _X^{p-1} )}{|\lambda |^j}.
\end{eqnarray*}
The $S_n$ are well-defined positive closed $(p-1,p-1)$ currents, because it is well-known (see for example Chapter 2 in \cite{guedj}) that
\begin{eqnarray*}
\lim _{n\rightarrow\infty}||(f^n)^*(\omega _X^{p-1} )||^{1/n}=\delta _{p-1}(f),
\end{eqnarray*}
and the latter is $< |\lambda |$ by assumption. The above inequality also shows that $||S_n||\rightarrow 0$ as $n\rightarrow \infty$. The $dd^c$-$p$
stability condition shows that $f^{\sharp}(dd^cR_n)$ is well-defined for any $n$, and moreover $f^{\sharp}(dd^cR_n)-\lambda dd^cR_{n+1}=-\lambda
dd^c\beta =\lambda dd^c\alpha $. Applying Theorem \ref{TheoremContinuityPropertiesOfPullbackOperator}, using that $R_n$ weakly converges to $R_{\alpha}$,
we have
\begin{eqnarray*}
f^{\sharp}(dd^cR_{\alpha})-\lambda dd^cR_{\alpha }=\lambda dd^c\alpha .
\end{eqnarray*}
\end{proof}
\begin{proof} (Of Theorem \ref{TheoremInvariantCurrents})

All of the equivalences follow easily from Lemma \ref{LemmaSolvingTheDdcForPullback}.

$1)\Rightarrow 3)$: Let $T_0$ be a closed $(p,p)$ current of order $s$ with $\{T_0\}=\theta _{\lambda}$ so that $f^{\sharp}(T_0)$ is well-defined, and
$f^{\sharp}(T_0)-\lambda T_0=0$. For any smooth $(p-1,p-1)$ form $\alpha$ on $X$, let $R_{\alpha}$ be the current constructed in Lemma
\ref{LemmaSolvingTheDdcForPullback}. Then $T=T_0+dd^c(R_{\alpha })$ is a closed $(p,p)$ current of order $s$ with $\{T\}=\theta _{\lambda}$ so that
$f^{\sharp}(T)$ is well-defined, and $f^{\sharp}(T)-\lambda T=dd^c(R_{\alpha})$.

$3)\Rightarrow 2$: Obviously.

$2)\Rightarrow 1)$: Let $\alpha _0$ be a smooth $(p-1,p-1)$ form, and let $T_0$ be a closed $(p,p)$ current of order $s$ with $\{T_0\}=\theta _{\lambda}$
so that $f^{\sharp}(T_0)$ is well-defined, and $f^{\sharp}(T_0)-\lambda T_0=dd^c(\alpha _0)$. Let $R_{\alpha}$ be the current constructed in Lemma
\ref{LemmaSolvingTheDdcForPullback}. Then $T=T_0-dd^c(R_{\alpha })$ is a closed $(p,p)$ current of order $s$ with $\{T\}=\theta _{\lambda}$ so that
$f^{\sharp}(T)$ is well-defined, and $f^{\sharp}(T)-\lambda T=0$.

Finally, that 2) and 4) are equivalent follows from the $dd^c$ lemma, since the current $f^{\sharp}(T)-\lambda T$ is a smooth form cohomologous to $0$.
\end{proof}
Now we give the proofs of Lemma \ref{LemmaInvariantCurrentForKahlerClass} and Corollaries \ref{CorollaryInvariantCurrentsForNonAlgebraicStability},
\ref{TheoremInvariantMeasureForLargeTopologicalDegree} and \ref{TheoremInvariantCurrentsForHyperpolicCohomologyHolomorphicMaps}.

\begin{proof} (Of Lemma \ref{LemmaInvariantCurrentForKahlerClass})
Since $\pi _1(\mathcal{C}_f)$ has codimension $\geq p$, it follows from Theorem \ref{TheoremInterestingExample1} any positive closed $(p,p)$ current can
be pulled back, and the pullback operator is continuous with respect to the weak topology on positive closed $(p,p)$ currents. We can represent $\theta$
by a difference $\alpha =\alpha ^+-\alpha ^-$ of two positive closed smooth $(p,p)$ forms $\alpha ^{\pm}$. Since $f$ is $p$-analytic stable, it follows
that $(f^n)^{*}(\alpha ^{\pm})=(f^{\sharp})^n(\alpha ^{\pm})$ are positive closed $(p,p)$ currents for any $n\geq 1$. Moreover there is a constant
$C_1>0$ so that $||(f^{\sharp})^n(\alpha ^{\pm})|| =||(f^n)^*(\alpha ^{\pm})||\leq C_1 r_p(f)^n=C_1\lambda ^n$ (see e.g \cite{guedj}). We follow the
standard construction of an invariant current under these assumptions (see \cite{sibony} and \cite{buff}). Consider the currents $T_N=T_N^+-T_N^-$, where
\begin{eqnarray*}
T_N^{\pm}=\frac{1}{N}\sum _{j=0}^{N-1}\frac{(f^{\sharp})^j(\alpha ^{\pm})}{\lambda ^j}.
\end{eqnarray*}

Then $T_N^{\pm}$ are positive closed $(p,p)$ currents with uniformly bounded masses, thus after passing to a subsequence, we may assume that they
converge to $T^{\pm}$. We define $T=T^+-T^-$. Since $\{T_N\}=\{\alpha\}$ for any $N$, we also have $\{T\}=\{\alpha\}$. Since
$f^{\sharp}(T_N^{\pm})-\lambda T_N^{\pm}$ converges to $0$, it follows that $f^{\sharp}(T)=\lambda T$.
\end{proof}

\begin{proof} (Of Corollary \ref{CorollaryInvariantCurrentsForNonAlgebraicStability})

Let $T_1$ and $T_2$ be the Green $(1,1)$ currents for the maps $f_1$ and $f_2$ as constructed in Sibony \cite{sibony}, respectively. Then we can write
\begin{eqnarray*}
T_i=\sum _{j}\lambda _{j,i}[V_{j,i}]
\end{eqnarray*}
for $i=1,2$, where $\lambda _{j,i}>0$ and $V_{j,i}$ are irreducible hypersurfaces in $\mathbb{P}^{k_i}$. Moreover $f^*(T_1)=d_1T_1$ and
$f^*(T_2)=d_2T_2$. We choose $T=T_1\times T_2$. Consider the finite summands
\begin{eqnarray*}
S_{N,i}=\sum _{j=0}^N\lambda _{j,i}[V_{j,i}].
\end{eqnarray*}
Then $f^{-1}(S_{N,1}\times S_{N,2})=f_1^{-1}(S_{N,1})\times f_2^{-1}(S_{N,2})$ has codimension $2$ in $\mathbb{P}^{k_1}\times \mathbb{P}^{k_2}$, thus
$f^{\sharp}(S_{N,1}\times S_{N,2})$ are well-defined by Corollary \ref{CorollaryPullBackVariety}. Since $T_1\times T_2-S_{N,1}\times S_{N,2}$ are
positive closed currents decreasing to $0$, it follows by Theorem \ref{TheoremContinuityPropertiesOfPullbackOperator} that $f^{\sharp}(T_1\times T_2)$ is
well-defined and moreover
$$f^{\sharp}(T_1\times T_2)=\lim _{N\rightarrow\infty}f^{\sharp}(S_{N,1}\times S_{N,2}).$$
It remains to show that $f^{\sharp}(T_1\times T_2)=d_1d_2T_1\times T_2$. To this end, first we show that $f^{\sharp}(S_{N,1}\times
S_{N,2})=f_1^{*}(S_{N,1})\times f_2^*(S_{N,2})$ for any $N$. By the results in \cite{dinh-sibony4} (see also the last section), there are positive closed
$(1,1)$ currents $W_{j,N,1}$ on $\mathbb{P}^{k_1}$ and $W_{j,N,2}$ on $\mathbb{P}^{k_2}$ with uniformly bounded norms so that $S_{N,1}=\lim
_{j\rightarrow\infty}W_{j,N,1}$ and $S_{N,2}=\lim _{j\rightarrow\infty}W_{j,N,2}$. Moreover, we can choose these approximations in such a way that
support of $W_{j,N,1}$ converges to $S_{N,1}$ and support of $W_{j,N,2}$ converges to $S_{N,2}$. Then $\lim _{j\rightarrow\infty}W_{j,N,1}\times
W_{j,N,2}=S_{N,1}\times S_{N,2}$, and $W_{j,N,1}\times W_{j,N,2}$ has uniformly bounded mass and locally uniformly converges to $0$ on
$\mathbb{P}^{k_1}\times \mathbb{P}^{k_2}-S_{N,1}\times S_{N,2}$. Hence we can apply Theorem \ref{TheoremInterestingExample} to obtain that
\begin{eqnarray*}
f^{\sharp}(S_{N,1}\times S_{N,2})&=&\lim _{j\rightarrow\infty}f^*(W_{j,N,1}\times W_{j,N,2})=\lim _{j\rightarrow\infty}f_1^*(W_{j,N,1})\times
f_2^*(W_{j,N,2})\\
&=&f_1^*(S_{N,1})\times f_2^*(S_{N,2}).
\end{eqnarray*}
Having this, it follows from the continuity of pullback on positive closed $(1,1)$ currents and the definitions of $T_1$ and $T_2$ that
\begin{eqnarray*}
f^{\sharp}(T_1\times T_2)&=&\lim _{N\rightarrow\infty}f^{\sharp}(S_{N,1}\times S_{N,2})=\lim _{N\rightarrow\infty}f_1^*(S_{N,1})\times f_2^*(S_{N,2})\\
&=&f^*(T_1)\times f^*(T_2)=d_1d_2T_1\times T_2.
\end{eqnarray*}
\end{proof}

\begin{proof} (Of Corollary \ref{TheoremInvariantMeasureForLargeTopologicalDegree})

It is well-known that for any smooth $(k,k)$ form $\theta$ then $(f^n)^{\sharp}(\theta )=(f^{\sharp})^n(\theta )$ (see for example \cite{guedj1}). Hence
$f$ satisfies $dd^c$-$k$ stability condition. As in \cite{guedj1}, we can find a smooth probability measure $\theta$ so that $f^*(\theta )$ is again a
smooth probability measure. Hence $f^*(\theta )-\delta _k(f)\theta =dd^c(\varphi )$, where $\varphi$ is a smooth $(p-1,p-1)$ form. Hence we can apply
Theorem \ref{TheoremInvariantCurrents}.
\end{proof}
\begin{proof} (Of Corollary \ref{TheoremInvariantCurrentsForHyperpolicCohomologyHolomorphicMaps})

Let $\theta$ be a smooth form then $f^*(\theta )$ is again a smooth form since $f$ is holomorphic. Then we can use the same arguments as that in the
proof of Corollary \ref{TheoremInvariantMeasureForLargeTopologicalDegree}.
\end{proof}

\section{Examples of good approximation schemes, and open questions}
We give some examples of good approximation schemes in Definition \ref{DefinitionGoodApproximation} in the first two subsections, and then discuss some open problems in the last subsection.
\subsection{The case of general K\"ahler manifolds}

Let $Z$ be a compact K\"ahler manifold of dimension $k$. Let $\pi _1,\pi _2:Z\times Z\rightarrow Z$ be the two projections, and let $\Delta _Z\subset Z\times Z$ be the diagonal.  Our construction of examples use the following regularization theorem of $DSH$ currents  in \cite{dinh-sibony1}.
\begin{theorem}
There is a sequence of strongly positive closed $(k,k)$ forms $K_n^{\pm}$ on $Z\times Z$ of $L^1$ coefficients with the following properties:

i) $K_n^+-K_n^-\rightharpoonup [\Delta _Z]$, and $||K_n^{\pm}||$ are uniformly bounded. The singularities of $K_n^{\pm}$ are the same as that of the
Bochner-Martinelli kernel.

ii) Support of $K_n^{+}-K_n^{-}$ converges to $\Delta _Z$. By this we mean, for any open neighborhood $U$ of $\Delta _Z$, there exists $n_0$ so that if
$n\geq n_0$ then support of $K_n^{+}-K_n^{-}$ is contained in $U$.

iii) If $T$ is a $DSH^p$ current then $(K_n^+-K_n^-)\wedge \pi _2^*(T)\rightharpoonup [\Delta _Z]\wedge \pi _2^*(T)$. Moreover, if $T_j$ converges to $T$
in $DSH^p(Z)$, then for a given number $n$: $(\pi _1)_*(K_n^{\pm}\wedge \pi _2^*(T_j))\rightharpoonup (\pi _1)_*(K_n^{\pm}\wedge \pi _2^*(T))$ when
$j\rightarrow \infty$.

Define $K_n^{\pm}(T)=(\pi _1)_*(K_n^{\pm}\wedge \pi _2^*(T))$, and $K_n(T)=K_n^+(T)-K_n^-(T)$. Then $K_n(T) \rightharpoonup  T$ in $DSH^p(Z)$ as
$n\rightarrow \infty$. Moreover, $||K_n^{\pm}(T)||_{DSH}\leq A||T||_{DSH}$, where $A>0$ is independent of $T$ and $n$.

iv) For any $s>0$, there exists a number $l_0=l_0(s)$ so that $K_{n_l}\circ K_{n_{l-1}}\circ \ldots \circ K_{n_1}(T)$ is a $C^s$ form for any $l\geq
l_0$, any integers $n_1,n_2,\ldots ,n_l$, and any $DSH^p$ current $T$.

v) If $T$ is a continuous form then $K_n(T)$ converges uniformly to $T$. \label{TheoremApproximationOfDinhAndSibony}\end{theorem}
\begin{proof}

The definition of $K_n^{\pm}$ is given in Section 3 in \cite{dinh-sibony1}, and we will recall the construction later in this subsection. All of the references below are from the same paper

i) is given in Lemma 3.1.

ii) is given in Remark 4.5.

iii) is given in Theorems 1.1 and 4.4.

iv) is given in Lemma 2.1.

v) is given in Proposition 4.6.
\end{proof}
Let us mention some notations  used later on.

\begin{remark} We use the following notations:

For integers $n_1,\ldots ,n_l$ and a $DSH^p$ current or continuous $(p,p)$ form $T$ on $Y$, we define $K_{n_l,n_{l-1},\ldots ,n_1}(T)=K_{n_l}\circ
K_{n_{l-1}}\circ \ldots \circ K_{n_1}(T)$. For simplicity, we write $(l)$ instead of $(n_1,\ldots ,n_l)$, and $\mathcal{K}_{(l)}(T)$ instead of
$K_{n_l,\ldots ,n_1}(T)$.

We write
\begin{eqnarray*}
\lim _{(n_1,n_2,\ldots ,n_l)\rightarrow \infty}T_{n_1,\ldots ,n_l}=T
\end{eqnarray*}
if for any sequence $(n_1)_k,\ldots ,(n_l)_k\rightarrow \infty$ we have
 \begin{eqnarray*}
\lim _{k\rightarrow \infty}T_{(n_1)_k,\ldots ,(n_l)_k}=T.
\end{eqnarray*}
 For simplicity we use
\begin{eqnarray*}
\lim _{(l)}T_{(l)}=T
\end{eqnarray*}
for such a limit.
\label{Remark1}\end{remark}

{\bf Example 4:} By Theorem \ref{TheoremApproximateInALinearWay} below, if $T$ is a $DSH$ current then
\begin{eqnarray*}
\lim _{(l)}\mathcal{K}_{(l)}(T)=T,
\end{eqnarray*}
for any $l\geq 0$.

The following consequence of Theorem \ref{TheoremApproximationOfDinhAndSibony} will be used to approximate $DSH^p(Y)$ currents by $C^s$ forms in a linear
way
\begin{theorem}
i) If $T_1$ is a $DSH^p(Z)$ current and $T_2$ is a continuous $(dim(Z)-p,dim(Z)-p)$ form on $Z$ then
\begin{eqnarray*}
\int _ZK_n^{\pm}(T_1)\wedge T_2=\int _ZT_1\wedge K_n^{\pm}(T_2).
\end{eqnarray*}

ii) For any integer $l$ and any $DSH^p(YZ$ current $T$, $\mathcal{K}_{(l)}(T)$ converges in $DSH^p(Z)$ to $T$. Here the convergence is understood in the
sense of Remark \ref{Remark1}. \label{TheoremApproximateInALinearWay}\end{theorem}

\begin{proof} (Of Theorem \ref{TheoremApproximateInALinearWay})

i) By Theorem \ref{TheoremApproximationOfDinhAndSibony}, the LHS of the equality we want to prove is continuous for the DSH convergence w.r.t. $T_1$. By
Lemma \ref{LemmaConvergenceOfDSHCurrents}, the RHS of the equality is also continuous for the DSH convergence w.r.t. $T_1$. Hence using the approximation
theorem for DSH currents of Dinh and Sibony, it suffices to prove the equality when $T_1$ is a smooth form, in which case it is easy to be verified.

ii) Note that since $||\mathcal{K}_{(l)}(T)||_{DSH}\leq A^l ||T||_{DSH}$ by Theorem \ref{TheoremApproximationOfDinhAndSibony}, to prove ii) it suffices
to show that $\mathcal{K}_{(l)}(T)$ converges weakly to $T$ in the sense of currents.

We prove by induction on $l$. If $l=1$, ii) is the content of Theorem \ref{TheoremApproximationOfDinhAndSibony}. To illustrate the idea of the proof, we
show for example how to prove ii) for the case $l=2$ when knowing ii) for $l=1$. Hence we need to show that: For a smooth $(dim (Z)-p,dim (Z)-p)$ form
$\alpha $
\begin{eqnarray*}
\lim _{(2)}\int _ZK_{n_2}\circ K_{n_1}(T)\wedge \alpha =\int _ZT\wedge \alpha .
\end{eqnarray*}
Since $\alpha$ is smooth, by i) we have
\begin{eqnarray*}
\lim _{(2)}\int _ZK_{n_2}\circ K_{n_1}(T)\wedge \alpha =\lim _{(2)}\int _ZK_{n_1}(T)\wedge K_{n_2}(\alpha ).
\end{eqnarray*}
By the case $l=1$ we know that $K_{n_1}(T)$ converges to $T$ in $DSH^p$. By Theorem \ref{TheoremApproximationOfDinhAndSibony}, $K_{n_2}(\alpha )$
converges uniformly to $\alpha$. Hence $\alpha -K_{n_2}(\alpha )$ is bound by $\epsilon _{n_2}\omega _Z^{dim(Z)-p}$, where $\epsilon _{n_2}\rightarrow 0$
as $n_2\rightarrow \infty$. A similar argument to that of the proof of Lemma \ref{LemmaConvergenceOfDSHCurrents} shows that
\begin{eqnarray*}
|\int _ZK_{n_1}(T)\wedge K_{n_2}(\alpha )-\int _ZK_{n_1}(T)\wedge \alpha |\leq A \epsilon _{n_2},
\end{eqnarray*}
where $A>0$ is independent of $n_1$ and $n_2$. Letting limit when $n_1,n_2$ converges to $\infty$ and using the induction assumption for $l=1$, we obtain
the claim for $l=2$.
\end{proof}

Now we define a good approximation scheme by $C^2$ forms as follows: Choose $l=l_0(2)$ in Theorem \ref{TheoremApproximationOfDinhAndSibony}, and choose the approximation $\mathcal{K}_{(2l)}$. Most of the requirements for good approximation scheme can be checked directly on $\mathcal{K}_{(2l)}$. The rest of this subsection shows the remaining requirements. The next remark concerns the $dd^c$ of $\mathcal{K}_{(2l)}$.

\begin{remark}
If $T$ is a $DSH^p(Y)$ current $T=T_1-T_2$ with $T_1,T_2$ positive $(p,p)$ currents, and $dd^c(T_i)=\Omega _i^+-\Omega
_i^-$ where $\Omega _i^{\pm}$ positive closed currents, then we can write:

i) $K_n(T)=T_{1,n}-T_{2,n}$ where $T_{1,n}=K_n^+(T_1)+K_n^-(T_2)$ and $T_{2,n}=K_n^-(T_1)+K_n^+(T_2)$ are positive currents with $L^1$ coefficients.

ii) $dd^c(T_{1,n})=\Omega _{1,n}^+-\Omega _{1,n}^{-}$, where $\Omega _{1,n}^+=K_n^+(\Omega ^+)+K_n^{-}(\Omega _2^+)$ and $\Omega _{1,n}^-=K_n^+(\Omega
^-)+K_n^{-}(\Omega _2^-)$ are positive closed $(p+1,p+1)$ currents with $L^1$ coefficients. Similarly, we can write $dd^c(T_{2,n})=\Omega _{2,n}^+-\Omega
_{2,n}^{-}$, where $\Omega _{2,n}^+$ and $\Omega _{2,n}^-$ are positive closed $(p+1,p+1)$ currents with $L^1$ coefficients.

iii) $||T_{i,n}||,||\Omega _{i,n}^{\pm}||\leq A||T||_{DSH}$, where $A>0$ is independent of $T$.

If we repeat this argument and use Theorem \ref{TheoremApproximationOfDinhAndSibony}, we see that for $l=2l_0(2)$ as above, we can write $\mathcal{K}_{(2l_0)}(T)=T_{1,(2l_0)}-T_{2,(2l_0)}$ where

i) $T_{i,(2l_0)}$ are positive $C^2$ forms.

ii) $dd^c(T_{i,(2l_0)})=\Omega _{i,(2l_0)}^+-\Omega _{i,(2l_0)}^-$, where $\Omega _{i,(2l_0)}^{\pm}$ are positive closed $C^2$ forms.

iii) $||T_{i,(2l_0)}||,||\Omega _{i,(2l_0)}^{\pm}||\leq A||T||_{DSH}$, where $A>0$ is independent of $T$.

More explicitly, we can write $\mathcal{K}_{(2l_0)}=\mathcal{K}^{+}_{(2l_0)}-\mathcal{K}^{-}_{(2l_0)}$, where $\mathcal{K}^{\pm}_{(2l_0)}$ are convex
combinations of compositions of $K_{m}^{\pm}$ (here $m$ belongs to the set $n_1,n_2,\ldots ,n_{2l_0}$), so that if $T$ is a positive $DSH$ current, then
$\mathcal{K}^{\pm}_{(2l_0)}(T)$ are positive currents. For example, if $l_0=1$, then $\mathcal{K}^{+}_{(2)}=K_{n_2}^+\circ K_{n_1}^++K_{n_2}^-\circ
K_{n_1}^-$ and $\mathcal{K}^{-}_{(2)}=K_{n_2}^-\circ K_{n_1}^++K_{n_2}^+\circ K_{n_1}^-$. Then we define
\begin{eqnarray*}
T_{1,(2l_0)}&=&\mathcal{K}^{+}_{(2l_0)}(T_1)+\mathcal{K}^{-}_{(2l_0)}(T_2),\\
T_{2,(2l_0)}&=&\mathcal{K}^{+}_{(2l_0)}(T_2)+\mathcal{K}^{-}_{(2l_0)}(T_1),\\
\Omega _{1,(2l_0)}^+&=&\mathcal{K}^{+}_{(2l_0)}(\Omega _1^+)+\mathcal{K}^{-}_{(2l_0)}(\Omega _2^+),\\
\Omega _{1,(2l_0)}^-&=&\mathcal{K}^{+}_{(2l_0)}(\Omega _1^-)+\mathcal{K}^{-}_{(2l_0)}(\Omega _2^-),
\end{eqnarray*}
and similarly for $\Omega _{2,(2l_0)}^{\pm}$.
 \label{Remark3}\end{remark}

The following refinements of Proposition 4.6 in \cite{dinh-sibony1} concern the continuity property of $\mathcal{K}_{(2l)}$. Its proof uses explicitly the properties of the kernels $K_n$ in Theorem \ref{TheoremApproximationOfDinhAndSibony} from Section 3 in \cite{dinh-sibony1}, which we recall briefly here. Let $\pi :\widetilde{Z\times Z}\rightarrow Z\times Z$ be the blowup along the diagonal $\Delta _Z$, and let $\widetilde{\Delta _Z}=\pi ^{-1}(\Delta _Z)$. Choose a strictly positive closed $(k-1,k-1)$ form $\gamma $ on $\widetilde{Z\times Z}$ so that $\pi _*(\gamma \wedge [\widetilde{\Delta _Z}])=[\Delta _Z]$. We let $\Theta '$ be a smooth closed $(1,1)$ form on $\widetilde{Z\times Z}$ having the same cohomology class with $[\widetilde{\Delta _Z}]$, and let $\varphi$ be a quasi PSH function so that $dd^c \varphi =[\widetilde{\Delta _Z}]-\Theta '$. Observe that $\varphi $ is smooth out of $[\widetilde{\Delta _Z}]$, and $\varphi ^{-1}(-\infty )=\widetilde{\Delta _Z}$. Let $\chi :\mathbb{R}\cup \{-\infty\} \rightarrow \mathbb{R}$ be a smooth increasing convex function such that $\chi (x)=0$ on $[-\infty ,-1]$, $\chi (x)=x$ on $[1, +\infty ]$, and $0\leq \chi '\leq 1$.  Define $\chi _n(x)=\chi (x+n)-n$, and $\varphi _n=\chi _n\circ \varphi $. The functions $\varphi _n$ are smooth decreasing to $\varphi$, and $dd^c \varphi _n\geq -\Theta $ for every $n$, where $\Theta$ is a strictly positive closed smooth $(1,1)$ form so that $\Theta -\Theta '$ is positive. Then we define $\Theta _n^+=dd^c \varphi _n+\Theta $ and $\Theta _n^-=\Theta ^-=\Theta -\Theta '$, and finally $K_n^{\pm}=\pi _*(\gamma \wedge \Theta _n^{\pm})$, and $K_n=K_n^{+}-K_n^-$.

\begin{proposition}
i) Let $T_n$ be a sequence of $DSH^p(Z)$ currents converging in $DSH$ to $T$. Assume that there is an open set $U\subset Z$ so that $T_n|_U$ are continuous forms, and $T_n$ converges locally uniformly on $U$ to $T$. Then $K_n^{\pm}(T_n)|_U$ are continuous and converges locally uniformly on $U$.

ii) Let $T$ be a $DSH^p(Z)$ current. Assume that there is an open set $U\subset Z$ so that $T|_U$ is a continuous form. Then for any positive integer $l$, $\mathcal{K}^{\pm}_{(l)}(T)|_{U}$ are continuous forms, and converges locally uniformly on $U$.
\label{PropositionUniformlyApproximation}\end{proposition}
\begin{proof}

i) Let $U_1\subset \subset U_2\subset\subset U_3\subset \subset  U$ be a relative compact open sets in $U$. We will show that $K_n^{\pm}(T_n)$ converges uniformly on $U_1$. Let $\chi _2:Z\rightarrow  [0,1]$ be a cutoff function for $U_2$ so that $\chi _2$ is smooth, $\chi _2=1$ on $U_2$ and $\chi _2=0$ outside of $U_3$. We write $K_n^{\pm}(T_n)=K_n^{\pm}(\chi _2T_n)+K_n^{\pm}((1-\chi _2 )T_n)$. By assumptions, $\chi _2T_n$ converges uniformly on $Z$ to $\chi _2T$, so there are $\epsilon _n$ decreasing to $0$ as $n\rightarrow 0$ so that $-\epsilon _n \omega _Z^p\leq \chi _2T_n-\chi _2T\leq \epsilon _n\omega _Z^p$. Then

\begin{eqnarray*}
-\epsilon _n K_n^{\pm}(\omega _Z^p)\leq K_n^{\pm}(\chi _2T_n)-K_n^{\pm}(\chi _2T)\leq \epsilon _n K_n^{\pm}(\omega _Z^p).
\end{eqnarray*}

Now $K_n^{-}(\omega _Z^p)= K^{-}(\omega _Z^p)$ is a smooth form, and hence $K_n^{+}(\omega _Z^p)=K_n(\omega ^p)-K^-(\omega _Z^p)$ is a sequence of smooth forms converging uniformly on $Z$, by applying  Proposition 4.6 in \cite{dinh-sibony1} to $\omega _Z^p$. Hence to prove i), it remains to show that  $K_n^{\pm}((1-\chi _2)T_n)$ converges uniformly on $U_1$.

We let $\chi _1:Z\rightarrow [0,1] $ be a cutoff function for $U_1$ so that $\chi _1$ is smooth, $\chi _1=1$ on $U_1$ and $\chi _1=0$ outside of $U_2$. Then it suffices to show that $\chi _1K_n^{\pm}((1-\chi _2)T_n)$ uniformly converges on $Z$.By definition, we have
\begin{eqnarray*}
\chi _1K_n^{\pm}((1-\chi _2)T_n)(x)&=&\int _Z\chi _1(x)K_n^{\pm}(x,y)\wedge (1-\chi _2(y))T_n(y)dy\\
 &=&\int _Z\chi _1(x)(1-\chi _2(y))K_n^{\pm}(x,y)\wedge T_n(y)dy .
\end{eqnarray*}
By definition of $\chi _1$ and $\chi _2$, the support of $\chi _1(x)(1-\chi _2(y))K_n^{\pm}(x,y)$ is contained in a fixed compact set of $Z\times Z-\Delta _Z$. Hence by definition of $K_n^{\pm}$, there is an $n_0$ and smooth forms $k ^{\pm}(x,y)$ on $Z\times Z$ so that $\chi _1(x)(1-\chi _2(y))K_n^{\pm}(x,y)=k^{\pm}(x,y)$ for all $n\geq n_0$. Then for $n\geq n_0$ we have
\begin{eqnarray*}
\chi _1K_n^{\pm}((1-\chi _2)T_n)(x)=\int _Zk^{\pm}(x,y)\wedge T_n(y)dy,
\end{eqnarray*}
 and the $RHS$ converges uniformly to $\int _Zk^{\pm}(x,y)\wedge T(y)dy$ since $T_n\rightharpoonup T$.

ii) We prove the claim for example for the case $l=1$ and $l=2$.

First, consider the case $l=1$. Then ii) follows by applying i) to the constant sequence $T_n=T$.

Now we consider the case $l=2$. Then $\mathcal{K}^{+}_{(2)}(T)=K_{n_2}^+\circ K_{n_1}^+(T)+K_{n_2}^-\circ K_{n_1}^-(T)$, and $\mathcal{K}_{(2)}^-(T)=K_{n_2}^+\circ K_{n_1}^-(T)+K_{n_2}^-\circ K_{n_1}^+(T)$. We show for example that $K_{n_2}^+\circ K_{n_1}^+(T)$ converges uniformly locally on $U$ as both $n_1$ and $n_2$ go to $\infty$. We apply i) to the sequence $T_n=K_n^{+}(T)$. The two conditions of i) are not hard to check: First, by the case $l=1$ the sequence $T_n$ converges locally uniformly on $U$. Second, by Theorem \ref{TheoremApproximateInALinearWay}, $T_n=K_n(T)+K^-(T)\rightharpoonup T +K^-(T)$.
\end{proof}

\subsection{The case of projective spaces}

In this case, Dinh and Sibony \cite{dinh-sibony4} used super-potential to define pullback of a positive closed current. We recall their definition in this subsection. The reader is referred to \cite{dinh-sibony4} for more detail.

a) Quasi-potentials:

Let $\omega$ be the Fubini-Study form on $\mathbb{P}^k$, normalized so that $||\omega ||=1$. Let $\mathcal{C}_p$ be the convex set of positive closed
$(p,p)$ currents $T$ on $\mathbb{P}^k$, normalized so that $||T||=1$. If $T\in \mathcal{C}_p$, then there is a $(p-1,p-1)$ current $U_T$ bounded from above so that
$T-\omega ^p=dd^c(U_T)$, and we call $m=\int _XU_T\wedge \omega ^{k-p+1}$ the mean of $U_T$. We call $U_T$ a quasi-potential of $T$ of mean $m$. For simplicity we choose $m=0$.

b) Deformation of currents:

The group $Aut(\mathbb{P}^k)$ of automorphisms of $\mathbb{P}^k$ is the complex Lie group $PGL(k+1,\mathbb{C})$ of dimension $k^2+2k$. We choose a local
holomorphic coordinate chart $y$ ($y\in \mathbb{C}^{k^2+2k}$, with $|y|<2$) of $Aut(\mathbb{P}^k)$ near the identity $id\in Aut(\mathbb{P}^k)$, in such a
way that $y=0$ at $id$. The element in $Aut(\mathbb{P}^k)$ with coordinate $y$ is denoted by $\tau _y$. Assume that the norm $|y|$ is invariant under the
involution $\tau \leftrightarrow \tau ^{-1}$. Choose a smooth probability $\rho$ with support in $|y|<1$ so that $\rho$ is radially and decreasing in
$|y|$.

Let $R$ be a positive or negative current on $\mathbb{P}^k$. For $\theta \in \mathbb{C}$ with $|\theta |\leq 1$, define
\begin{equation}
R_{\theta}:=\int _{Aut(\mathbb{P}^k)}(\tau _{\theta y})_*(R)d\rho (y)=\int _{Aut(\mathbb{P}^k)}(\tau _{\theta y})^*(R)d\rho
(y).\label{EquationDeformationOfCurrentsOnPk}
\end{equation}
This has the same positiveness or negativeness as $R$.  Lemma 2.1.5 in \cite{dinh-sibony4} shows that as $\theta \rightarrow 0$ then $R_{\theta}$ weakly
converges to $R$ and $supp(R_{\theta})$ converges to $supp(R)$. Moreover, if $U\subset \mathbb{P}^k$ is open and $R|_U$ is continuous, then $R_{\theta}$ converges locally uniformly on $U$ to $R$.

c) Super-potential:

Let $S$ be a smooth form in $\mathcal{C}_p$, and let $R$ be in $\mathcal{C}_{k-p+1}$. If $U_R$ is a quasi-potential of $R$ (of mean $0$), then the number
$\int _{X}S\wedge U_R$ is independent of the choice of $U_R$, and is denoted by
\begin{eqnarray*}
\mathcal{U}_S(R)=\int _{X}S\wedge U_R,
\end{eqnarray*}
and $\mathcal{U}_S$ is called the superpotential (of mean $0$) of $S$.

For arbitrary $S\in \mathcal{C}_p$ and $R\in \mathcal{C}_{k-p+1}$, define
\begin{eqnarray*}
\mathcal{U}_S(R):=\lim _{\theta \rightarrow 0}\mathcal{U}_{S_{\theta}}(R)=\lim _{\theta\rightarrow 0}\mathcal{U}_{S}(R_{\theta}).
\end{eqnarray*}
Note that this definition is symmetric $\mathcal{U}_S(R)=\mathcal{U}_R(S)$.

d) Pullback of currents:

Let $f:\mathbb{P}^k\rightarrow \mathbb{P}^k$ be a dominant rational map. A positive closed $(p,p)$ current $T$ is called $f^*$-admissible if
$$\mathcal{U}_T(f_*(\omega ^{k-p+1}))>-\infty .$$
In this case, we define $f^*(T)$ as follows:
\begin{eqnarray*}
f^*(T)=\lim _{\theta \rightarrow 0}f^*(T_{\theta}).
\end{eqnarray*}

\subsection{Some open questions}

Let $X$ be a compact K\"ahler manifold, and let $f:X\rightarrow X$ be a dominant meromorphic map.

A) Let $T$ be a positive closed $(p,p)$ current on $X$ with Siu's decomposition $T=R+\sum _{j}\lambda _j[V_j]$. Let $E(T)$ be as in Theorem
\ref{TheoremUseSiuDecomposition}. Assume that for any irreducible analytic $V$ contained in $E(T)$ then $f^{-1}(V)$ has codimension $\geq p$. Is
$f^{\sharp}(T)$ well-defined? If so, is $f^{\sharp}(R)$ positive? Note that by Corollary \ref{CorollaryTheMapJ}, $f^{\sharp}(T)$ may not be positive
though.

B) Assume that $\pi _1(\mathcal{C}_f)$ has codimension $\geq p$.

a) When $X=\mathbb{P}^k$, \cite{dinh-sibony4} showed that $\pi _1(\mathcal{C}_{f^n})$ has codimension $\geq p$ for all $n$. Is the same true for a
general $X$?

b) Does $f$ satisfy $dd^c$-$p$ stability condition? This holds for $p=1$.

c) Using a) and the fact that when $X=\mathbb{P}^k$ then $f^{\sharp}$ preserves the convex cone of positive $(p,p)$ currents, \cite{dinh-sibony4} showed
that if moreover $f$ is $p$-algebraic stable then $(f^n)^{\sharp}=(f^{\sharp})^n$ for all $n$. Does the same conclusion hold when $X$ is arbitrary K\"ahler
manifold? We check that the answer to this question is positive when $f=J_X$:

\begin{lemma}
Let $J_X$ be the same map in Section 4. Then $J_X$ is $2$-algebraic stable and $(J_X^{\sharp})^2=Id =(J_X^2)^{\sharp}$ on positive closed $(2,2)$ currents.
\label{LemmaJXSquare}\end{lemma}
\begin{proof}
Since $J_X$ has no exceptional hypersurface, $J_X$ is $1$-algebraic stable. Because $J_X=J_X^{-1}$, it follows by duality that $J_X$ is also $2$-algebraic stable. Since $J_X^2=Id$,  it remains to check that $(J_X^{\sharp})^2=Id$.  Define $A=\bigcup _{i\not= j}\widetilde{\Sigma}_{i,j}$.

1) First we show that for a $DSH^1$ current $R$ then:
\begin{equation}
(J_X^{\sharp})^2(R)=R.\label{EquationLemmaJXSquare.1}
\end{equation}
 For this end, first we show that $(J_X^{\sharp})^2(R)=R$ on $X-A$. Since $J_X^{\sharp}$ is continuous in the $DSH ^1$ topology by Theorem \ref{TheoremInterestingExample1}, using Theorem \ref{TheoremApproximationOfDinhAndSibony}  it suffices to show (\ref{EquationLemmaJXSquare.1}) for a smooth $(1,1)$ form $R$. In that case it is easy to see, since $(J_X^{\sharp})^2(R)$ is determined by its restriction on $X-A$, and on $X-A$ it is not other than the usual pullback of smooth forms  $(J_X|_{X-A}^*)^2(R)$.

 Having $(J_X^{\sharp})^2(R)=R$ on $X-A$, then (\ref{EquationLemmaJXSquare.1}) follows by the Federer type of support in \cite{{bassanelli}}.

2) It follows from 1) that if $T$ is a positive closed $(2,2)$ current on $X$, then $(J_X^{\sharp})^2(T)-T$ depends only on the cohomology class of $T$. In fact, if $T'$ is a positive closed $(2,2)$ current having the same cohomology class as $T$, then $T-T'=dd^c(R)$ for a $DSH^1$ current $R$. Then from 1)
\begin{eqnarray*}
(J_X^{\sharp})^2(T)-(J_X^{\sharp})^2(T')=dd^c(J_X^{\sharp})^2(R)=dd^c(R)=T-T'.
\end{eqnarray*}

3) From 2), to prove Lemma \ref{LemmaJXSquare} it suffices to show it for a set of positive closed currents whose cohomology classes generate
$H^{2,2}(X)$. For such a set, we can consider the currents of integrations on a generic line in $\mathbb{P}^3$, a generic line in the exceptional
divisors $E_0,E_1,E_2,E_3$, and the line $\widetilde{\Sigma}_{i,j}$. In these cases, the wanted equality is easy to be checked.
\end{proof}

C) Can the constructions of invariant currents in the Subsection \ref{SubsectionApplications} be extended to other cases, for example for a map in
Question B?

Lemma \ref{LemmaInvariantCurrentForKahlerClass} gives a positive support to this question. More generally, for any meromorphic map $f$, there are natural
candidates $\mu$ for an invariant measure of $f$. These measures can be standardly constructed as in the proof of Lemma
\ref{LemmaInvariantCurrentForKahlerClass}: Let $\alpha$ be a smooth probability measure. Then $\mu$ is a cluster point of the sequence
\begin{eqnarray*}
\mu _N=\frac{1}{N}\sum _{j=0}^{N-1}\frac{(f^{*})^j(\alpha )}{\delta _k(f)^j}.
\end{eqnarray*}
There are two problems remain to be solved. First, we don't know whether the measure $\mu$ constructed this way can be pulled back or not. Second, we
don't know whether we have a continuity property to help showing that $f^{\sharp}(\mu )=\delta _k(f)\mu $. If we can extend Theorem
\ref{TheoremContinuityPropertiesOfPullbackOperator} to be applicable to the sequence $\mu _N$ then we can solve these two problems altogether.


\end{document}